\documentclass[11pt]{article}
\usepackage{amsthm}
\usepackage{amsmath}
\usepackage{amssymb}
\usepackage[toc,page]{appendix}
\usepackage{color,enumitem}
\usepackage{bbm}

\allowdisplaybreaks

\makeatletter
\@addtoreset{equation}{section}

\makeatother

\makeatletter

\newdimen\mainfontsize \mainfontsize=1\@ptsize pt
\topskip=\mainfontsize
   \@maxdepth=\maxdepth

\makeatother

\theoremstyle{plain}
\newtheorem{thm}{Theorem}[section]
\newtheorem{lem}[thm]{Lemma}

\theoremstyle{definition}
\newtheorem{defn}[thm]{Definition}
\newtheorem{exmp}[thm]{Example}

\theoremstyle{remark}
\newtheorem{rem}[thm]{Remark}
\newtheorem*{claima}{\bf Claim A}
\newtheorem*{claimb}{\bf Claim B}

\title{Well-posedness of a system of SDEs driven by jump random measures}
\author{Ying Jiao\footnote{Universit\' e Claude Bernard - Lyon 1, Institut de Science Financière et d'Assurances, 50 Avenue Tony Garnier, 69007 Lyon, France. Email: ying.jiao@univ-lyon1.fr.} $\,$ and Nikolaos Kolliopoulos\footnote{Peking University, Beijing International Centre for Mathematical Research, 100871 Beijing,  China.} \footnote{Carnegie Mellon University, Department of Mathematical Sciences, Pittsburgh, PA 15213, USA. Email: nkolliop@andrew.cmu.edu (corresponding author).} }

\begin{document}
\maketitle

%
%
%
%
%
%

\begin{abstract}
We establish well-posedness for a class of systems of SDEs with non-Lipschitz coefficients in the diffusion and jump terms and with two sources of interdependence: a monotone function of all the components in the drift of each SDE and the correlation between the driving Brownian motions and jump random measures. Pathwise uniqueness is derived by employing some standard techniques. Then, we use a comparison theorem along with our uniqueness result to construct non-negative, $L^1$-integrable c\`adl\`ag solutions as monotone limits of solutions to approximating SDEs, allowing for time-inhomogeneous drift terms to be included. Our approach allows also for a comparison property to be established for the solutions to the systems we investigate. The applicability of certain systems in financial modeling is also discussed. 

\end{abstract}

\section{Introduction}
In this paper, we are interested in {systems of stochastic differential equations (SDEs) with non-Lipschitz diffusion and jump terms and two sources of interdependence: one in the drifts of the equations and one in the correlated driving processes. These systems include multi-dimensional generalisations with mean-field interactions in the drift of the following one-dimensional jump SDE}  
\begin{equation}\label{lambda-root}
d\lambda_t = a \left( b  - \lambda_t  \right) dt + \sigma \sqrt{\lambda_t} dB_t
+\sigma_Z \sqrt[\alpha]{\lambda_{t-}}  dZ_t,  \quad t\geq 0 \end{equation}
where $a,b,\sigma,\sigma_Z\geq 0$, $B=(B_t,t\geq 0)$  is a Browinan motion and $Z=(Z_t,t\geq 0)$ is an independent spectrally positive $\alpha$-stable compensated L\'evy process with parameter $\alpha \in (1,2]$. The existence of unique strong solutions to jump SDEs with generally non-Lipschitz coefficients is obtained by Fu and Li \cite{FL2010} (see also Li and Mytnik \cite{LM2011}). Dawson and Li \cite{DL2006} consider and prove a more general integral representation of \eqref{lambda-root} in terms of CBI processes (continuous state branching processes with immigration) 
\begin{equation}\label{lambda-integral}
\lambda_t = \lambda_0 + a\int_0^t  \left( b  - \lambda_s  \right) ds + \sigma \int_0^t \int_0^{\lambda_s} W(ds,du)
+ \sigma_Z \int_0^t \int_0^{\lambda_{s-} } \int_{\mathbb{R}^+} \zeta \widetilde{N}
(ds,dv, d\zeta),
\end{equation}
where $W(ds,du)$ is a white noise on $\mathbb{R}_+^2$  with intensity $dsdu$ and $\widetilde N(ds,dv,d\zeta)$ is an independent compensated Poisson 
random measure on $\mathbb{R}_+^3$ with intensity $dsdv\mu(d\zeta)$, where $\mu(d\zeta)$ is a L\'evy measure on $\mathbb{R}_+$ satisfying  
$\int_0^\infty (\zeta\wedge\zeta^2)\mu(d\zeta)<\infty $.

The process given by \eqref{lambda-root} generalises the well-known 
Cox–Ingersoll–Ross (CIR) process and its applications in mathematical finance are studied by Jiao et al. \cite{JCS2017, JCSZ2021}.
The link between general CBI processes and the affine modeling framework is established by Filipovi\'c  \cite{Fil2001}. In a recent paper, Frikha and Li \cite{FL2021} studied the well-posedness and numerical approximation of a time-inhomogeneous jump SDE with generally non-Lipschitz coefficients and a drift term that involves the law of the solution, which can be viewed as the large-population limit of an individual particle evolving within a system with mean-field interaction (McKean-Vlasov SDE).  On one hand, assuming generally non-Lipschitz coefficients makes the well-posedness challenging since the classic iteration method fails to 
apply (see \cite{FL2010,DL2006,FL2021}). On the other hand, CIR-like processes are non-negative and ergodic so they constitute good candidates for financial modelling, see Fouque et al. \cite{FOU} for the very nice properties of some volatility models of this kind. A system of correlated SDEs with CIR volatility and absorption at $0$ has been studied by Hambly and Kolliopoulos \cite{HK2017,ERR, HK2019} in the framework of portfolio credit modeling.

In this paper, we focus on a system of finite number of jump SDEs where the drift term of each component {depends monotonically on other components of the system and constitutes one of the two sources of interaction, with the second one being the correlation among the driving processes, that is, the Brownian motions and jump random measures}. Such {systems of} SDEs with diffusion coefficients being non-Lipschitz {and a mean-field component in the drifts} can be adopted to model interacting financial quantities. We refer to works by Bo and Capponi \cite{BC2015}, Fouque and Ichiba \cite{FI2013} and Giesecke et al. \cite{GSS2013} for the mean-field modelling of large default-sensitive portfolios in credit and systemic risks. {In our setup the SDEs} contain a jump part driven by general random measures which allows to include various jump processes such as Poisson or compound Poisson processes, L\'evy processes or indicator default processes. 

The solutions to the equations constituting our multi-dimensional system are obtained as monotone limits of solutions to increasing sequences of approximating SDEs with with {an arbitrary c\`adl\`ag process $(b_t)_{t \in \left[0, T\right]}$ present in the drift}. The key element is an auxiliary lemma that establishes existence of solutions and a comparison property for the approximating equations. The proof of the auxiliary lemma is based on further monotone approximations by simpler SDEs where $(b_t)_{t \in \left[0, T\right]}$ is replaced by some piecewise constant process and thus the results of \cite{FL2010} become applicable, with the monotonicity and the comparison property derived by applying the comparison theorem of Gal'chuk \cite{Y1986} (see also Abdelghani and Melnikov). Establishing the comparison property requires also the limits of approximating sequences to be unique, which is a consequence of the uniqueness of solutions to the general multi-dimensional system. The latter is obtained in a second auxiliary lemma, the proof of which is almost identical to that of the uniqueness result in \cite{FL2010}. The solutions to our system are also positive since they are obtained as monotone limits of other positive processes. Finally, our method covers also systems with time-inhomogeneous drifts, where a martingale problem approach cannot be adopted, while it establishes also a comparison property for the unique solutions to those systems.

The rest of the paper is organized as follows. In Section \ref{sec:system}, we present our system of SDEs, the assumptions on the coefficients, {and an example of a possible application in financial modeling}. The main well-posedness result and its proof are given in Section \ref{sec:main result}.
The proofs of the auxiliary lemmas are left to 
the Appendix in Section \ref{sec: key lemma}.

\section{System of SDEs and assumptions}\label{sec:system}

We fix a filtered probability space $\left(\Omega, \mathcal{F}, \mathbb F= (\mathcal{F}_{t})_{t \geq 0}, \mathbb{P} \right)$ 
which satisfies the usual conditions. Let $U_0$ and $U_1$ be two locally compact and separable metric spaces, in which case they will also be $\sigma$-compact spaces. For $N\in\mathbb N$, we study the following system of SDEs
\begin{eqnarray}
\lambda_t^i &=& \lambda_0^i + a_i\int_{0}^{t}\left(b_{i}\left(s, \, \lambda_s^1, \, \lambda_s^2, \, ..., \, \lambda_s^N\right) - \lambda_s^i\right)ds + \int_0^t\sigma_i(\lambda_s^i)dW_s^i \nonumber \\
&& \,\quad + \int_0^{t}\int_{U_0}g_{i,0}\left(\lambda_{s-}^i, \, u\right)\tilde{N}_{i,0}\left(ds, \, du\right) \nonumber \\
&& \,\quad  + \int_0^{t}\int_{U_1}g_{i,1}\left(\lambda_{s-}^i, \, u\right)N_{i,1}\left(ds, \, du\right), \quad i \in \{1, 2, ..., N\}\label{mainsystem}
\end{eqnarray}
where  $\lambda_0^i\geq 0$, $a_i\geq 0$,  
$W^i=(W^i_t)_{t\geq 0}$ is an $\mathbb{F}$-adapted standard Brownian motion, $N_{i,0}\left(ds, \, du\right)$ and $N_{i,1}\left(ds, \, du\right)$ are the Poisson measures of {two} $\mathbb{F}$-adapted point processes 
$p_{i,0}: \Omega \times \mathbb R_+ \longrightarrow U_0$ and $p_{i,1}: \Omega \times \mathbb R_+ \longrightarrow U_1$
with compensator {measures} $\mu_{i,0}(du)dt$ and $\mu_{i,1}(du) dt$ respectively, and $\tilde{N}_{i,0}\left(ds, \, du\right) = N_{i,0}\left(ds, \, du\right) - \mu_{i,0}(du)dt$ is the compensated Poisson measure of $p_{i,0}$. Finally, we assume that $W^i,N_{i,0},N_{i,1}$ are independent for each fixed $i$, but we do not require $(W^i,N_{i,0},N_{i,1})$ to be mutually independent for different $i \in \{0, 1, ..., N\}$.  
\begin{exmp}\label{exmp}
Let $Z^i=(Z^i_t)_{t\geq 0}$ for $i \in \{0, \, 1, \, ..., \, N\}$ be pairwise independent spectrally positive $\alpha_i$-stable L\'evy processes with $\alpha_i\in(1,2]$ for each $i$, and $B^i = (B^i_t)_{t\geq 0}$ for $i \in \{0, \, 1, \, ..., \, N\}$ be pairwise independent standard Brownian motions which are also independent from each L\'evy process $Z^i$. Then, we have the following multi-dimensional generalization of \eqref{lambda-root}:
\begin{eqnarray}\label{examplesystem}
\lambda_t^i = \lambda_0^i &+& a_i\int_{0}^{t}\frac{1}{N}\sum_{j=1}^N\left(\lambda_s^j - \lambda_s^i\right)ds + \sigma_{i}\int_0^t\sqrt{\lambda_s^i}dB_s^i + \sigma_{0}\int_0^t\sqrt{\lambda_s^i}dB_s^0 \nonumber \\
&+& \sigma_{Z, i}\int_0^{t}\sqrt[\alpha_i]{\lambda_{s-}^i}dZ^i_s + \sigma_{Z, 0}\int_0^{t}\sqrt[\alpha_0]{\lambda_{s-}^i}dZ^0_s, \qquad i \in \{1, \, 2, \, ..., \, N\} \nonumber \\
\end{eqnarray}
where $a_i, \, \sigma_i, \, \sigma_{Z, i} \geq 0$ for all $i \in \{1, \, 2, \, ..., \, N\}$, $\sigma_0, \, \sigma_{Z, 0} \geq 0$, and $\alpha_0, \alpha_i \in (1,2] $. It is not hard to see that \eqref{examplesystem} is obtained as a special case of \eqref{mainsystem} where, for each $i \in \{1, \, 2, \, ..., \, N\}$, the drift function $b_i$ is given by $$b_i\left(s, \, x_1, \, x_2, \, ..., \, x_N\right)=\frac{1}{N}\sum_{j=1}^Nx_j$$ for all $(s, \, x_1, \, x_2, \, ..., \, x_N) \in \mathbb{R}_{+} \times \mathbb{R}^N$, $\tilde{N}_{i,0}(dt,d{u})$ for ${u}=(u_0,u_i)\in \mathbb{R}_{+}^2$ is the compensated L\'evy measure associated with the two-dimensional Levy process $\overline{Z}^i=(Z^0,Z^i)$, the Brownian motion $W^i$ is obtained as a linear combination of $B^i$ and $B^0$, i.e $$W^i = \frac{\sigma_i}{\sqrt{\sigma_i^2 + \sigma_0^2}}B^i + \frac{\sigma_0}{\sqrt{\sigma_i^2 + \sigma_0^2}}B^0,$$ and for $x \geq 0$ we have $$\sigma_i(x)= \sqrt{\sigma_i^2 + \sigma_0^2}\cdot x^{\frac{1}{2}}$$ and $$g_{i,0}(x, u)=\sigma_{Z,0}\cdot x^{\frac{1}{\alpha_0}} \cdot u_0+\sigma_{Z,i}\cdot x^{\frac{1}{\alpha_i}} \cdot u_i,$$ with $\sigma_i$, $\sigma_0$, $\sigma_{Z,0}$ and $\sigma_{Z,i}$ being positive real numbers.
Since each process $\lambda^i$ has only positive jumps which usually describe shocks due to negative phenomena, the above system is suitable for modeling default intensities, where the mean-field term in the drift of each equation captures interacting activities (e.g lending) which pull the corresponding intensity towards the market average. Both $B^0$ and $Z^0$ represent external factors which affect the whole market, while $B^i$ and $Z^i$ for $i \geq 1$ capture movements in the $i$-th intensity due to idiosyncratic factors. {The introduction of the jump external factor $Z^0$ allows for shocks due to major external events such as financial crisis or pandemics (like the recent Covid pandemic crisis) to be captured, as well as more common events which lead to more frequent and smaller jumps. Similarly, the processes $Z^i$ for $i \geq 1$ are associated to idiosyncratic shocks that lead to individual jumps. Observe also that jump terms of the more general form \eqref{lambda-integral} can also be included if $g_0$ takes the form $g_0(x, v,\zeta)=1_{\{v< x\}}\zeta$ where $u=(v,\zeta)\in\mathbb R_+^2$.}
\end{exmp}

We now impose some necessary conditions on the diffusion and jump terms of our system \eqref{mainsystem}, which will be crucial for proving our result. We start with the following definition

\begin{defn}\label{assumptions1}
We denote by $\operatorname{Assum1}(\sigma,g_0,g_1,N_0,N_1)$ the conjunction of the following conditions.
\begin{enumerate}[label=(\arabic*)]

    \item $\sigma:\mathbb{R}\rightarrow\mathbb R$ is a continuous function which is either bounded or increasing on $\mathbb{R}_{+}$, and it holds that $\sigma(x) = 0$ for $x \leq 0$.

    \item $N_0$ is the Poisson random measure of an $\mathbb{F}$-adapted point process $p_0$ with a locally finite characteristic measure $\mu_0$, the function $g_0:\mathbb R\times U_0\rightarrow\mathbb R$ is Borel-measurable, and the following are satisfied:
    \begin{enumerate}[label=(\roman*)]
    \item for each fixed $u\in U_0$, {the function} $g_0(\cdot,u):x\mapsto g_0(x,u)$ is increasing and satisfies the inequality $g_0(x,u)+x\geq 0$ when $x\geq 0$ and the equality $g_0(x,u)=0$ when $x\leq 0$,
    \item the function
    $x\longmapsto\int_{U_0}|g_0(x,u)|\wedge|g_0(x,u)|^2\,\mu_0(du)$
    is locally bounded,

    \end{enumerate}
    \item $N_1$ is the Poisson random measure of an $\mathbb{F}$-adapted point process $p_1$ with a characteristic measure $\mu_1$, the function $g_1:\mathbb R\times U_1\rightarrow\mathbb R$ is Borel-measurable, and the following are satisfied:
    \begin{enumerate}[label=(\roman*)]
    \item for each fixed $u\in U_1$, {the function} $g_1(\cdot,u):x\mapsto g_1(x,u)$ is increasing and satisfies the inequality $g_1(x,u)+x\geq 0$ for every $x \in \mathbb{R}$ 
    \item there exists a Borel set $U_2 \subset U_1$ with $\mu_{1}\left(U_1 \backslash U_2\right) < +\infty$ such that the function
    \begin{eqnarray}\label{bigjumpsbound}
     x\longrightarrow\int_{U_2}\left|g_{1}(x,u)\right|\mu_{1}(du)\end{eqnarray}
    grows at most linearly as $x\rightarrow+\infty$,
    \end{enumerate}
    
    \item There exist positive and increasing functions $x \longrightarrow \rho(x)$ and $x \longrightarrow r_{m}(x)$ for each $m \in \mathbb{N}$, all defined on $\mathbb R_+$, with
    \begin{equation}\label{equ: holder like}
    \int_{0}^{x}\frac{dz}{\rho^2(z)} = +\infty \quad \text{and} \quad \int_{0}^{x}\frac{dz}{r_{m}(z) + z} = +\infty \quad \text{for} \quad m \in \mathbb{N}
    \end{equation}
    for any $x > 0$, such that:
    \begin{enumerate}[label=(\roman*)]
    \item $|\sigma(x)-\sigma(y)| \leq \rho(|x-y|)$ for all $x,y\geq 0$.
    \item $\left|g_{j}\left(x,u\right) - g_{j}\left(y,u\right)\right| \leq \rho(|x-y|)G_j(u)$ for all $x,y\geq 0$ and $u \in U_j$, for a function $u \longrightarrow G_j(u)$ defined on $U_j$ with
    \begin{eqnarray}
    \int_{U_j} G_j^2(u)\mu_j(du) < +\infty,   
    \end{eqnarray}
    for $j \in \{0, 1\}$.
    \item For each $m \in \mathbb{N}$, $r_m(\cdot)$ is a concave function and it holds that
    \begin{eqnarray}\label{lastest}
    \int_{U_2}\left|g_{1}\left(x,u\right)\wedge m - g_{1}\left(y,u\right)\wedge m\right|\mu_{1}(du) \leq r_{m}(|x - y|) 
    \end{eqnarray}
    for all $0 \leq x, \, y \leq m$, where $U_2$ is the same set as in (ii) of (3).
    \end{enumerate}

    \item For each $\bar{x} \geq 0$, there exist $U_{0,\bar{x}} \subset U_0$ and $U_{1,\bar{x}} \subset U_1$ such that for $j \in \{0, 1\}$ we have that $\mu_j(U_j \backslash U_{j,\bar{x}}) = 0$ and the function $x \longrightarrow g_{j}(x, u_j)$ is continuous at $x = \bar{x}$ for all $u_j \in U_{j,\bar{x}}$.
\end{enumerate}
\end{defn}


\begin{rem}\label{assumrm2}
This paper builds on the existence result \cite[Theorem 5.1]{FL2010}, which we will see that it only needs to be applied on jump SDEs with linear reversion towards a positive mean in the drift, i.e equations of the form
\begin{eqnarray}\label{SDEform}
x(t) = x(0) &+& a\int_{0}^{t}\left(b - x(s)\right)ds + \int_0^t\sigma(x(s))dW_s \nonumber \\
&+& \int_0^{t}\int_{U_1}g_{1}\left(x(s-), \, u\right)N_{1}\left(ds, \, du\right) + \int_0^{t}\int_{U_0}g_{0}\left(x(s-), \, u\right)\tilde{N}_{0}\left(ds, \, du\right). \nonumber \\
\end{eqnarray}
We can verify that the conditions required by \cite{FL2010} for the above are all contained in $\operatorname{Assum1} (\sigma,g_{0},g_{1},N_{0},N_{1})$. In particular:
\begin{itemize}
\item (2.b) from \cite{FL2010} follows immediately from the boundedness and/or monotonicity of the function $\sigma$ and the monotonicity of the function $g_0$.
\item Since $U_0$ is separable and locally compact, it is also a $\sigma$-compact space, so since $\mu_0$ is a locally finite measure, we can verify (5.b) from \cite{FL2010} by taking $V_1, \, V_2, \, \ldots$ to be compact subsets of $U_0$ with $\cup_{i=1}^nV_i = U_0$ and by applying Holder's inequality on the integral in (ii) of (2) in $\operatorname{Assum1} (\sigma,g_{0},g_{1},N_{0},N_{1})$.
\item Given that $U_0$ is separable and locally compact, the converse of the above is also true: if (5.b) from \cite{FL2010} is satisfied, any compact subset $K$ of $U_0$ is contained in some $V_i$, so we have $\mu_0(K) \leq \mu_0(V_i) < \infty$ and thus $\mu_0$ is locally finite. This indicates that when $U_0$ is a finite-dimensional Banach space, which will be the case in most natural applications, our assumptions are not much stronger than those in \cite{FL2010}. 
\end{itemize}
\end{rem}

\begin{rem}\label{assumrm3}
The second main tool used in this paper is the comparison theorem of Gal’chuk \cite[Theorem 1]{Y1986}, which needs to be applied on pairs of jump SDEs of the form \eqref{SDEform} that differ only at the mean $b$ and at the initial condition $x(0)$, that is
\begin{eqnarray}\label{SDEsforcompar}
x_i(t) = x_i(0) &+& a\int_{0}^{t}\left(b_i - x_i(s)\right)ds + \int_0^t\sigma(x_i(s))dW_s \nonumber \\
&+& \int_0^{t}\int_{U_0}g_{0}\left(x_i(s-), \, u\right)\tilde{N}_{0}\left(ds, \, du\right) + \int_0^{t}\int_{U}g_{1}\left(x_i(s-), \, u\right)N_{1}\left(ds, \, du\right), \nonumber \\
\end{eqnarray}
for $i \in \{1,2\}$, with $b_2 > b_1$ and $x_2(0) \geq x_1(0)$. Equations \eqref{SDEsforcompar} can be brought to the form of equation (1) in \cite{Y1986} with $X^i = x_i(\cdot)$, driving processes $a_t = t$ and $m_{t} = W_{t}$ for all $t \geq 0$, 
the counting measure $\mu$ defined as
\begin{eqnarray}
\mu([0, t], \cdot) = \sum_{0 \leq s \leq t}\mathbbm{1}_{\{(\Delta_s p_0, \, \Delta_s p_1) \neq (0, 0)\}}\delta_{(\Delta_s p_0, \, \Delta_s p_1)}(\cdot) \quad \text{for} \quad \Delta_s p_j = p_j(s) - p_j(s-)\nonumber
\end{eqnarray}
with intensity
\begin{eqnarray}
\nu(dt, du) = dt \times (\delta_{0}(du_0)\times \mu_1(du_1) + \mu_0(du_0) \times \delta_{0}(du_1)) \quad \text{for} \quad u = (u_0, u_1), \nonumber
\end{eqnarray}
the coefficient functions given by
\begin{eqnarray}
f^i(X^i) &=& a(b_i - X^i) + \int_{U_1}\mathbbm{1}_{\{u_1 \leq 1\}}g_1(X^i, u_1)\mu_1(du_1) - \int_{U_0}\mathbbm{1}_{\{u_0 > 1\}}g_0(X^i, u_0)\mu_0(du_0), 
\nonumber \\
g(X^i) &=& \sigma(X^i), \nonumber \\
h(u, X^i) &=& h^i(u, X^i) \,\,\,\, = \,\,\,\, \mathbbm{1}_{\{u_0 \neq 0\}}g_0(X^i, u_0) + \mathbbm{1}_{\{u_1 \neq 0\}}g_1(X^i, u_1) \quad \text{for} \quad u = (u_0, u_1), \nonumber
\end{eqnarray}
and finally the measure $p$ (not to be confused with the point processes $p_0$ and $p_1$ defined in this paper) being the zero measure so we can take e.g {$\ell^1(X^1) = 1$}, $\ell^2(X^2) = 2$ and $k^1(X^1) = k^2(X^2) = 0$. We can now verify that the conditions required by \cite{Y1986} for applying the comparison theorem to deduce that $x_2(t) \geq x_1(t)$ for all $t$ are also contained in $\operatorname{Assum1} (\sigma,g_{0},g_{1},N_{0},N_{1})$. Specifically:
\begin{itemize}
\item The continuity of $H_2$ in \cite{Y1986} can be derived from (5) of $\operatorname{Assum1} (\sigma,g_{0},g_{1},N_{0},N_{1})$, while $H_3$ and $H_4$ are obtained from (i) and (ii) of (4) in $\operatorname{Assum1} (\sigma,g_{0},g_{1},N_{0},N_{1})$ with
\[H(s,u) = \mathbbm{1}_{\{u_0 \neq 0\}}G_0(u_0) + \mathbbm{1}_{\{u_1 \neq 0\}}G_1(u_1) \quad \text{for} \quad u = (u_0, u_1)\]
\item $H_5$ and $H_6$ of \cite{Y1986} are satisfied by the monotonicity conditions in (i) of (2) and (3) in $\operatorname{Assum1} (\sigma,g_{0},g_{1},N_{0},N_{1})$ and by the choice of $k^i$ and $\ell^i$, while the conditions in $H_7$ are straightforward to verify.
\item To verify $H_8$ of \cite{Y1986} with $\rho(\cdot)$ given in (4) of $\operatorname{Assum1} (\sigma,g_{0},g_{1},N_{0},N_{1})$, we observe first that the sequence $\{a_n\}_{n \in \mathbb{N}}$ can be constructed due to \eqref{equ: holder like}. Next, considering uncountably many pairwise disjoint perturbations of $\{a_n\}_{n \in \mathbb{N}}$ which have the same property as the original sequence, some will consist of continuity points of $\rho(\cdot)$ since monotone functions have countably many discontinuities, which means that we can take each $a_n$ to be a continuity point of $\rho(\cdot)$. Then, it is sufficient to pick each term of $\{\epsilon_n\}_{n \in \mathbb{N}}$ sufficiently close to $0$.
\end{itemize}
\end{rem}

\begin{rem}
The conditions imposed on $\sigma$ and $g_0$ by (4) of $\operatorname{Assum1} (\sigma,g_{0},g_{1},N_{0},N_{1})$ can be compared to the H\" older condition in \cite{FL2021}. They are satisfied if $\sigma$ and $g_0$ are (uniformly) $\alpha$-H\" older continuous in $x$ for some $\alpha\in[1/2, 1]$. In that case, one can take $\rho(x) = \sqrt[\alpha]{x}$ which is also continuous, meaning that the complex argument used to verify the condition $H_8$ of \cite{Y1986} in the previous remark is not needed. Observe also that the systems of Example~\ref{exmp} fall into this case.
\end{rem}

\begin{rem}\label{assumrm1}
The conditions imposed by (5) of $\operatorname{Assum1} (\sigma,g_{0},g_{1},N_{0},N_{1})$ are also satisfied by the systems in Example~\ref{exmp}. In particular, the case $g_0(x, u)= g_0(x, v, \zeta) = 1_{\{v< x\}}\zeta$ is covered since for each $\bar{x} \in \mathbb{R}_{+}$ there is a jump at $x = \bar{x}$ only when $v = \bar{x}$.
\end{rem}

\noindent Next, the drift functions $b_i$ must also satisfy some continuity and growth conditions. These are given in the following definition

\begin{defn}\label{assumptionsfinal} For a function $b: \mathbb{R}_{+} \times \mathbb{R}^N \longrightarrow \mathbb{R}$, we say that $\operatorname{Assum2} (b)$ is satisfied when:
\begin{enumerate}
\item For any $(x_1, \, x_2, \, \ldots, \, x_N) \in \mathbb{R}_{+}^N$, the function $b(\cdot, \, x_1, \, x_2, \, \ldots, \, x_N): \mathbb{R}_{+} \longrightarrow \mathbb{R}$ is c\`adl\`ag (right continuous with left limits).
\item For any $t \geq 0$, the function $b(t, \cdot, \, \cdot, \, \ldots, \, \cdot): \mathbb{R}^N \longrightarrow \mathbb{R}$ is Lipschitz continuous and increasing in each of its $N$ arguments.
\item For any $T \geq 0$, there exist constants $B_T, L_T > 0$ such that $$b(t, \, x_1, \, x_2, \, \ldots, \, x_N) \leq B_T + L_T(x_1 + x_2 + \ldots + x_N)$$ for all $t \in [0, \, T]$ and all $(x_1, \, x_2, \, \ldots, \, x_N) \in \mathbb{R}_{+}^N$.
\end{enumerate}
\end{defn}

\section{Existence and pathwise uniqueness of the solution}
\label{sec:main result}
The main result of this paper is given in the following Theorem.

\begin{thm}\label{thm:1}Consider the system of SDEs \eqref{mainsystem} and suppose that for every index $i \in \{1,\cdots, N\}$ the following conditions hold: 
\begin{enumerate}[label=\rm(\arabic*)]
\item  $a_i > 0$ and $b_i(t, \, x_1, \, \ldots, \, x_N) \geq 0$ for all $(t, \, x_1, \, \ldots, \, x_N) \in \mathbb{R}_{+}^{N+1}$.
\item $\operatorname{Assum1} (\sigma_i,g_{i,0},g_{i,1},N_{i,0},N_{i,1})$ and $\operatorname{Assum2}(b_i)$ are satisfied.
\end{enumerate}
Then, \eqref{mainsystem} has a pathwise unique $\mathbb{F}$-adapted c\`adl\`ag solution $(\lambda^1_t,\cdots,\lambda^N_t)_{t\geq 0}$, with each $\lambda^i_{\cdot}$ being non-negative and satisfying the integrability condition $\mathbb E[\int_0^T\lambda_t^i\,dt]<+\infty$ for any $T\geq 0$. Moreover, denoting by $(\lambda^{1}_{j,t},\cdots,\lambda^{N}_{j,t})_{t\geq 0}$ the unique $\mathbb{F}$-adapted c\`adl\`ag solution to \eqref{mainsystem} but with each $b_i$ replaced by a different function $b_{i,j}$, where for $i \in \{1, \, \ldots, \, N\}$ and $j \in \{1, 2\}$ we have
\[
b_{i,1}(t, \, x_1, \, \ldots, \, x_N) \geq b_{i,2}(t, \, x_1, \, \ldots, \, x_N) \quad \text{for all} \quad (t, \, x_1, \, \ldots, \, x_N) \in \mathbb{R}_{+}^{N+1} \]
and
\[\lambda_{1,0}^{i} \geq \lambda_{2,0}^{i},
\]
it holds that $\lambda_{1,t}^{i} \geq \lambda_{2,t}^{i}$ for all $t \geq 0$ and $i \in \{1, \, \ldots, \, N\}$.
\end{thm}

\noindent The proof of the existence result contained in the above theorem relies heavily on the existence of solutions and a comparison property for a class of one-dimensional SDEs, where the proof of the latter requires the pathwise uniqueness of solutions. Therefore, we start by presenting the uniqueness result as a first lemma, and then we present the main auxiliary lemma which contains the existence result and the comparison property for the one-dimensional SDEs.

\begin{lem}[Only Uniqueness]\label{Uniqthm}
Under the assumptions of Theorem~\ref{thm:1}, there is at most one $\mathbb{F}$-adapted c\`adl\`ag solution $\left(\lambda_{\cdot}^1, \, \lambda_{\cdot}^2, \, ..., \, \lambda_{\cdot}^N\right)$ to the system \eqref{mainsystem}.
\end{lem}

\begin{lem}\label{basiclemma}
Let $T>0$. Consider the SDE
\newcommand*{\myTagFormat}[2]{$(\ref{#1})_{#2}$}
\refstepcounter{equation}\label{auxSDE}
\begin{equation}\begin{aligned}
Y_t = Y_0 &+ a\int_{0}^{t}\left(b_s - Y_s\right)ds + \int_0^t\sigma(Y_s)dW_s \nonumber \\
&+ \int_0^{t}\int_{U_1}g_{1}\left(Y_{s-}, \, u\right)N_{1}\left(ds, \, du\right)  + \int_0^{t}\int_{U_0}g_{0}\left(Y_{s-}, \, u\right)\tilde{N}_{0}\left(ds, \, du\right)
\end{aligned}
\tag*{\myTagFormat{auxSDE}{b}}
\label{auxSDEb}
\end{equation}
for $t \in \left[0, \, T\right]$, where $a>0$ and $b=(b_t)_{t\in[0,T]}$ is a non-negative $\mathbb{F}$-adapted c\`adl\`ag process. Suppose that $\operatorname{Assum1} (\sigma,g_0,g_1,N_0,N_1)$ is satisfied by the components $(\sigma,g_0,g_1,N_0,N_1)$. Then, $\eqref{auxSDE}_{b}$ has a non-negative $\mathbb{F}$-adapted c\`al\`ag solution $Y=(Y_{t})_{t \in \left[0, \, T\right]}$. Moreover, if $b^1=(b^1_t)_{t\in[0,T]}$ and $b^2=(b^2_t)_{t\in[0,T]}$ are two non-negative $\mathbb{F}$-adapted c\`adl\`ag processes and $(Y^i)_{t \in [0, \, T]}$ is a solution to $\eqref{auxSDE}_{b^i}$ for $i \in \{1, 2\}$, the conditions $Y_0^1 \geq Y_0^2$ and $b_t^1 \geq b_t^2$ for all $t \in [0, \, T]$ (almost surely) imply that $$Y^1_t \geq Y^2_t$$ for all $t \in [0, \, T]$ (also almost surely).
\end{lem}

\begin{rem}
 Note that in the above equation $\eqref{auxSDE}_{b}$, the symbol ``$b$'' is attached as a subscript to its label in order to emphasize the dependence of the equation on the drift coefficient process $b$.  The process $b$ could be replaced by  some particular processes in the following and the subscript will be changed accordingly.   
\end{rem}


We proceed now to the proof of Theorem \ref{thm:1}, were Lemmas~\ref{Uniqthm} and \ref{basiclemma} will be used. 

\begin{proof}[Proof of Theorem 3.1] Without loss of generality, we show that the equation admits a solution $(\lambda^1_t,\cdots,\lambda^N_t)$ for ${t\in [0,T]}$, with $\lambda^i_t$ non-negative and $\mathbb E[\int_0^T\lambda_t^i\,dt]<+\infty$ for a given $T> 0$. Then the solution can be extended to $\mathbb R_+$ without difficulty.  \bigskip  

\noindent{\bf Step 1.} {\it Construction of the approximating systems and monotonicity.} For  $n \in \mathbb{N}$, we construct a partition $0 = t^n_0 < t^n_1 < \ldots < t^n_{2^{n-1}} = T$ of $\left[0, \, T\right]$ as follows: We start with $t^1_0 = 0$ and $t^1_1 = T$ and, for any integer $n$,  define inductively $t^{n+1}_{2j} = t^n_j$ for all $j\in\{0,\ldots,2^{n-1}\}$ and $t^{n+1}_{2j+1} = {(t^n_{j} + t^n_{j+1})}/{2}$ for all $j\in\{0,\ldots,2^{n-1}-1\}$. Next, for each $i \in \{1, 2, ..., N \}$, let $\lambda_{\cdot}^{i,1}$ to be the non-negative solution to the SDE
\begin{eqnarray}
\lambda_t^{i,1} = \lambda_0^{i} &-&a_i\int_{0}^{t}\lambda_s^{i,1}ds + \int_0^t\sigma_i(\lambda_s^{i,1})dW_s^i \nonumber \\
&+& \int_0^{t}\int_{U_1}g_{i,1}\left(\lambda_{s-}^{i,1}, \, u\right)N_{i,1}\left(ds, \, du\right)  + \int_0^{t}\int_{U_0}g_{i,0}\left(\lambda_{s-}^{i,1}, \, u\right)\tilde{N}_{i,0}\left(ds, \, du\right) \nonumber \\ \label{initialsystem} 
\end{eqnarray}
which exists by Theorem 5.1 in \cite{FL2010} (see also Remark~\ref{assumrm2} regarding the verification of the required conditions). Then, having $\lambda_{\cdot}^{i,n}$ defined for some $n \geq 1$ and all $i \in \{1, 2, ..., N \}$, we define:
\begin{eqnarray}
b^{i,n}_k = \inf_{s \in \left[t^n_k, \, t^{n}_{k+1}\right]}b_i(s, \, \lambda_s^{1,n}, \, \lambda_s^{2,n}, \, ..., \, \lambda_s^{N,n})\label{mean}
\end{eqnarray}
and $\lambda_{\cdot}^{i,n+1}$ in $\left[t^n_k, \, t^n_{k+1}\right]$ for any $k \in \{0, 1, ..., 2^{n-1} - 1\}$ by solving the SDE
\begin{eqnarray}\label{approxSDE}
\lambda_t^{i,n+1} = &&\lambda_{t^n_k}^{i,n+1} + a_i\int_{t^n_k}^{t}\left(\mathbb{E}\left[b^{i,n}_k | \mathcal{F}_{s}\right] - \lambda_s^{i,n+1}\right)ds + \int_{t^n_k}^t\sigma_i(\lambda_s^{i,n+1})dW_s^i \nonumber \\
&& + \int_{t^n_k}^{t}\int_{U_1}g_{i,1}\left(\lambda_{s-}^{i,n+1}, \, u\right)N_{i,1}\left(ds, \, du\right)  + \int_{t^n_k}^{t}\int_{U_0}g_{i,0}\left(\lambda_{s-}^{i,n+1}, \, u\right)\tilde{N}_{i,0}\left(ds, \, du\right) \nonumber \\ \label{recursion}
\end{eqnarray}
for $t \in \left[t^n_k, \, t^n_{k+1}\right]$, which also has a solution by Lemma \ref{basiclemma}.

\noindent We will show now that for any $n \geq 1$ we have $\lambda_t^{i,n+1} \geq \lambda_t^{i,n}$ for all $i \in \{1, 2, ..., N\}$ and all $t \in \left[0, \, T\right]$ by {induction on $n$}. For the {base} case, that is, $\lambda_t^{i,2} \geq \lambda_t^{i,1}$, we only need to recall that $\mathbb{E}[b^{i,n}_k | \mathcal{F}_{s}] \geq 0$  since   each $b_i$ in \eqref{mean} is a non-negative function, and then use the comparison property of Lemma~\ref{basiclemma}. Suppose now that for some $n \geq 1$ we have $\lambda_t^{i,n+1} \geq \lambda_t^{i,n}$ for all $i \in \{1, 2, ..., N\}$ and $t \in \left[0, \, T\right]$. Then, by the monotonicity of each $b_i$ we have
\begin{eqnarray}
b^{i,n+1}_{2k} &=& \inf_{s \in \left[t^{n+1}_{2k}, \, t^{n+1}_{2k+1}\right]}b_i(s, \, \lambda_s^{1,n+1}, \, \lambda_s^{2,n+1}, \, ..., \, \lambda_s^{N,n+1}) \nonumber \\
&\geq& \inf_{s \in \left[t^{n+1}_{2k}, \, t^{n+1}_{2k+1}\right]}b_i(s, \, \lambda_s^{1,n}, \, \lambda_s^{2,n}, \, ..., \, \lambda_s^{N,n}) \nonumber \\
&\geq& \inf_{s \in \left[t^{n+1}_{2k}, \, t^{n+1}_{2(k+1)}\right]}b_i(s, \, \lambda_s^{1,n}, \, \lambda_s^{2,n}, \, ..., \, \lambda_s^{N,n}) \nonumber \\
&=& \inf_{s \in \left[t^{n}_{k}, \, t^{n}_{k+1}\right]}b_i(s, \, \lambda_s^{1,n}, \, \lambda_s^{2,n}, \, ..., \, \lambda_s^{N,n}) \nonumber \\
&=& b^{i,n}_{k} \label{evenmean}
\end{eqnarray}
for $n \geq 1$, all $i \in \{1, 2, ..., N\}$ and all $k \in \{0, 1, ..., 2^{n-1} - 1\}$, and also
\begin{eqnarray}
b^{i,n+1}_{2k+1} &=& \inf_{s \in \left[t^{n+1}_{2k+1}, \, t^{n+1}_{2k+2}\right]}b_i(s, \, \lambda_s^{1,n+1}, \, \lambda_s^{2,n+1}, \, ..., \, \lambda_s^{N,n+1}) \nonumber \\
&\geq& \inf_{s \in \left[t^{n+1}_{2k+1}, \, t^{n+1}_{2k+2}\right]}b_i(s, \, \lambda_s^{1,n}, \, \lambda_s^{2,n}, \, ..., \, \lambda_s^{N,n}) \nonumber \\
&\geq& \inf_{s \in \left[t^{n+1}_{2k}, \, t^{n+1}_{2k+2}\right]}b_i(s, \, \lambda_s^{1,n}, \, \lambda_s^{2,n}, \, ..., \, \lambda_s^{N,n}) \nonumber \\
&=& \inf_{s \in \left[t^{n}_{k}, \, t^{n}_{k+1}\right]}b_i(s, \, \lambda_s^{1,n}, \, \lambda_s^{2,n}, \, ..., \, \lambda_s^{N,n}) 
\nonumber \\
&=& b^{i,n}_{k} \label{oddmean}
\end{eqnarray}
for $n \geq 1$, all $i \in \{1, 2, \ldots, N\}$ and all $k \in \{0, 1, \ldots, 2^{n-1} - 1\}$. We will use these two inequalities to show that $\lambda_t^{i,n+2} \geq \lambda_t^{i,n+1}$ for all $i \in \{1, 2, \ldots, N\}$ and $t \in \left[0, \, T\right]$. This is done by applying a second induction as follows: For $t \in [t^{n+1}_{0}, \, t^{n+1}_{1}] =[0, \, t^{n+1}_{1}] \subset [0, \, t^{n}_{1}]$ we have
\begin{eqnarray}
\lambda_t^{i,n+2} = \lambda_{0}^{i,n+2} &+& a_i\int_{0}^{t}\left(\mathbb{E}\left[b^{i,n+1}_0 | \mathcal{F}_s \right] - \lambda_s^{i,n+2}\right)ds + \int_0^t\sigma_i(\lambda_s^{i,n+2})dW_s^i \nonumber \\
&+& \int_0^{t}\int_{U_1}g_{i,1}\left(\lambda_{s-}^{i,n+2}, \, u\right)N_{i,1}\left(ds, \, du\right) \nonumber \\
&+& \int_0^{t}\int_{U_0}g_{i,0}\left(\lambda_{s-}^{i,n+2}, \, u\right)\tilde{N}_{i,0}\left(ds, \, du\right)\label{comparbase1}
\end{eqnarray}
and
\begin{eqnarray}
\lambda_t^{i,n+1} = \lambda_{0}^{i,n+1} &+& a_i\int_{0}^{t}\left(\mathbb{E}\left[b^{i,n}_0 | \mathcal{F}_s\right] - \lambda_s^{i,n+1}\right)ds  + \int_0^t\sigma_i(\lambda_s^{i,n+1})dW_s^i \nonumber \\
&+& \int_0^{t}\int_{U_1}g_{i,1}\left(\lambda_{s-}^{i,n+1}, \, u\right)N_{i,1}\left(ds, \, du\right) \nonumber \\
&+& \int_0^{t}\int_{U_0}g_{i,0}\left(\lambda_{s-}^{i,n+1}, \, u\right)\tilde{N}_{i,0}\left(ds, \, du\right)\label{comparbase2}
\end{eqnarray}
and since $\mathbb{E}[b^{i,n+1}_{0} | \mathcal{F}_s] \geq \mathbb{E}[b^{i,n}_{0} | \mathcal{F}_s]$ (by taking conditional expectations in \eqref{evenmean} for $k=~0$ and with $n$ replaced with $n+1$) the comparison property of Lemma~\ref{basiclemma} implies that $\lambda_t^{i,n+2} \geq \lambda_t^{i,n+1}$ for every $t \in[t^{n+1}_{0}, \, t^{n+1}_{1}] = [0, \, t^{n+1}_{1}]$. Suppose now that for some $k' \in~\{0, 1, \ldots, 2^{n} - 1\}$ we have $\lambda_t^{i,n+2} \geq \lambda_t^{i,n+1}$ for all $t \in [t^{n+1}_{0}, \, t^{n+1}_{k'}] = \left[0, \, t^{n+1}_{k'}\right]$. Then for $k' = 2k$ with $k \in \{0, 1, ..., 2^{n - 1} - 1\}$ we have $t^{n+1}_{k'} = t^{n}_{k}$ and $t^{n+1}_{k'+1} = (t^{n}_{k} + t^{n}_{k+1})/{2}$, while for $k' = 2k+1$ with $k \in \{0, 1, ..., 2^{n - 1} - 1\}$ we have $t^{n+1}_{k'} = {(t^{n}_{k} + t^{n}_{k+1})}/{2}$ and $t^{n+1}_{k'+1} = t^{n}_{k+1}$, so in both cases it holds that $[t^{n+1}_{k'}, t^{n+1}_{k'+1}] \subset [t^{n}_{k}, t^{n}_{k+1}]$ and for any $t \in [t^{n+1}_{k'}, t^{n+1}_{k'+1}]$ we have both
\begin{eqnarray}
\lambda_t^{n+2} = \lambda_{t^{n+1}_{k'}}^{i,n+2} &+& a_i\int_{t^{n+1}_{k'}}^{t}\left(\mathbb{E}\left[b^{i,n+1}_{k'} | \mathcal{F}_s \right] - \lambda_s^{i,n+2}\right)ds + \int_{t^{n+1}_{k'}}^t\sigma_i(\lambda_s^{i,n+2})dW_s^i \nonumber \\
&+& \int_{t^{n+1}_{k'}}^{t}\int_{U_1}g_{i,1}\left(\lambda_{s-}^{i,n+2}, \, u\right)N_{i,1}\left(ds, \, du\right) \nonumber \\
&+& \int_{t^{n+1}_{k'}}^{t}\int_{U_0}g_{i,0}\left(\lambda_{s-}^{i,n+2}, \, u\right)\tilde{N}_{i,0}\left(ds, \, du\right)\label{comparinduct1}
\end{eqnarray}
and
\begin{eqnarray}
\lambda_t^{n+1} = \lambda_{t^{n+1}_{k'}}^{i,n+1} &+& a_i\int_{t^{n+1}_{k'}}^{t}\left(\mathbb{E}\left[b^{i,n}_{k} | \mathcal{F}_s \right] - \lambda_s^{i,n+1}\right)ds  + \int_{t^{n+1}_{k'}}^t\sigma_i(\lambda_s^{i,n+1})dW_s^i \nonumber \\
&+& \int_{t^{n+1}_{k'}}^{t}\int_{U_1}g_{i,1}\left(\lambda_{s-}^{i,n+1}, \, u\right)N_{i,1}\left(ds, \, du\right) \nonumber \\
&+& \int_{t^{n+1}_{k'}}^{t}\int_{U_0}g_{i,0}\left(\lambda_{s-}^{i,n+1}, \, u\right)\tilde{N}_{i,0}\left(ds, \, du\right)\label{comparinduct2}
\end{eqnarray}
with $\mathbb{E}\left[b^{i,n+1}_{k'} | \mathcal{F}_s\right] \geq \mathbb{E}\left[b^{i,n}_{k} | \mathcal{F}_s\right]$ (by taking expectations given $\mathcal{F}_s$ in \eqref{evenmean} and \eqref{oddmean}). Thus, the comparison theorem implies that $\lambda_t^{i,n+2} \geq \lambda_t^{i,n+1}$ for all $t \in \left[t^{n+1}_{k'}, \, t^{n+1}_{k'+1}\right]$, which means that the same inequality holds for all $t \in \left[t^{n+1}_{0}, \, t^{n+1}_{k'+1}\right] \equiv \left[0, \, t^{n+1}_{k'+1}\right]$. This completes the second induction and gives $\lambda_t^{i,n+2} \geq \lambda_t^{i,n+1}$ for all $t \in \left[0, \, T\right]$, and the last completes the initial induction giving $\lambda_t^{i,n+1} \geq \lambda_t^{i,n}$ for all $t \in \left[0, \, T\right]$ and all $n \geq 1$.\bigskip

\noindent{\bf Step 2.} {\it Finiteness of the monotone limits.} We have shown in the previous step that {the family of processes} $\{\lambda_t^{i,n}\}_{t \in [0,T]}$ is pointwise increasing in $n$, we will show that almost surely, $\displaystyle{\lim_{n\longrightarrow +\infty}\lambda_t^{i,n}}$ is finite for almost all $t \in \left[0, \, T\right]$ and every $i \in \{1,\,2,\,\ldots,\,N\}$. This will follow by Fatou's lemma if we can show that
\begin{eqnarray}
\sup_{1\leq i \leq N}\mathbb{E}\left[\int_0^T\lambda_t^{i,n}dt\right]
\end{eqnarray}
is bounded in $n \in \mathbb{N}$. For the last, recall that by $\operatorname{Assum2}(b_i)$ we have that
\begin{eqnarray}\label{lipineq}
b_i\left(s, \, \lambda_s^{1,n}, \, \lambda_s^{2,n}, \, ..., \, \lambda_s^{N,n}\right) \leq B_T + L_T\sum_{i=1}^{N}\lambda_s^{i,n}
\end{eqnarray}
for each $i \in \{1, \, 2, \, \ldots, \, N\}$ and $s \in \left[t^n_k, \, t^n_{k+1}\right] \subset [0, \, T]$, so replacing the LHS of the above by its infimum for $s \in \left[t^n_k, \, t^n_{k+1}\right]$ and then conditioning on $\mathcal{F}_{s}$ we obtain
\begin{eqnarray}
\mathbb{E}\left[b_k^{i,n}\, |\,\mathcal{F}_{s}\right] \leq B_T + L_T\sum_{i = 1}^{N}\lambda_s^{i,n}
\end{eqnarray}
for all $s \in \left[t^n_k, \, t^n_{k+1}\right]$ and $i \in \{1, \, 2, \, ..., \, N\}$. Plugging the above in \eqref{approxSDE}, localizing if needed, taking expectations and then supremum in $i$ and finally using \eqref{bigjumpsbound}, we can easily get
\begin{eqnarray}
\sup_{1 \leq i \leq N}\mathbb{E}\left[\lambda_t^{i,n+1}\right] \leq \sup_{1 \leq i \leq N}\mathbb{E}\left[\lambda_{t_k^n}^{i,n+1}\right] &+& \bar{a}B(t - t_k^{n}) + \bar{a}LN\int_{t_k^n}^t\sup_{1 \leq i \leq N}\mathbb{E}\left[\lambda_s^{i,n}\right]ds \nonumber \\
&+& K\int_{t_k^n}^t\left(\sup_{1 \leq i \leq N}\mathbb{E}\left[\lambda_s^{i,n}\right] + 1\right)ds
\end{eqnarray}
for $\displaystyle{\bar{a} := \sup_{1 \leq i \leq N}a_i}$, which can be written as
\begin{eqnarray}\label{lastsummand}
\sup_{1 \leq i \leq N}\mathbb{E}\left[\lambda_t^{i,n+1}\right] \leq \sup_{1 \leq i \leq N}\mathbb{E}\left[\lambda_{t_k^n}^{i,n+1}\right] &+& B'(t - t_k^{n}) + L'\int_{t_k^n}^t\sup_{1 \leq i \leq N}\mathbb{E}\left[\lambda_s^{i,n}\right]ds \nonumber \\
\end{eqnarray}
for $B' = \bar{a}B + K$ and $L' = \bar{a}LN + K$, so replacing $k$ with $k' < k$ and taking $t = t_{k'+1}^n$ we get also
\begin{eqnarray}\label{firstsummands}
\sup_{1 \leq i \leq N}\mathbb{E}\left[\lambda_{t_{k'+1}^n}^{i,n+1}\right] &\leq& \sup_{1 \leq i \leq N}\mathbb{E}\left[\lambda_{t_{k'}^n}^{i,n+1}\right] \nonumber \\
&& + \, B'(t_{k'+1}^n - t_{k'}^{n}) + L'\int_{t_{k'}^n}^{t_{k'+1}^n}\sup_{1 \leq i \leq N}\mathbb{E}\left[\lambda_s^{i,n}\right]ds.
\end{eqnarray}
Summing \eqref{lastsummand} with \eqref{firstsummands} for $k' \in \{0, \, 1, \, ..., \, k - 1\}$ we obtain
\begin{eqnarray}\label{be}
\sup_{1 \leq i \leq N}\mathbb{E}\left[\lambda_{t}^{i,n+1}\right] \leq \sup_{1 \leq i \leq N}\mathbb{E}\left[\lambda_{0}^i\right] + B't + L'\int_{0}^{t}\sup_{1 \leq i \leq N}\mathbb{E}\left[\lambda_s^{i,n}\right]ds
\end{eqnarray}
and since $k$ was arbitrary, the above holds for any $t \in \left[0, \, T\right]$. Take now a constant $M > 0$ such that $\displaystyle{\sup_{1 \leq i \leq N}\mathbb{E}\left[\lambda_{t}^{i,1}\right] \leq M}$ for all $t \in \left[0, \, T\right]$, which is possible by recalling the estimate $(2.5)$. Then, provided that $M$ is large enough, we will show by induction on $n$ that 
\begin{equation}\label{almostfinalbound}
\displaystyle{\sup_{1 \leq i \leq N}\mathbb{E}\left[\lambda_{t}^{i,n}\right] \leq Me^{L't}}
\end{equation}
for all $n \in \mathbb{N}$ and $t \in \left[0, \, T\right]$. The base case is trivial, and if $M$ is large enough such that $\displaystyle{M > \sup_{1 \leq i \leq N}\mathbb{E}\left[\lambda_{0}^i\right] + B'T}$, plugging $\displaystyle{\sup_{1 \leq i \leq N}\mathbb{E}\left[\lambda_{s}^{i,n}\right] \leq Me^{L's}}$ in \eqref{be} we find that
\begin{eqnarray}
\sup_{1 \leq i \leq N}\mathbb{E}\left[\lambda_{t}^{i,n+1}\right] &\leq& \sup_{1 \leq i \leq N}\mathbb{E}\left[\lambda_{0}^i\right] + B't + L'M\int_{0}^{t}e^{L's}ds \nonumber \\
&=& \sup_{1 \leq i \leq N}\mathbb{E}\left[\lambda_{0}^i\right] + B't + Me^{L't} - M 
\leq Me^{L't}
\end{eqnarray}
which completes the induction. Integrating then \eqref{almostfinalbound} for $t \in \left[0, \, T\right]$ we obtain the desired boundedness.\bigskip

\noindent {\bf Step 3.} {\it Showing that the limiting processes satisfy the system of SDEs.}
Now that we have the pointwise monotone convergence of the process $\{\lambda_{t}^{i,n}\}_{t \in \left[0, \, T\right]}$ to a finite process $\{\lambda_{t}^{i}\}_{t \in \left[0, \, T\right]}$ for every index $i \in \{1, \, 2, \, ..., \, N\}$, we will show that these limiting processes solve our system of SDEs. The first step is to fix an $i \in \{1, \, 2, \, ..., \, N\}$, and for each $n \in \mathbb{N}$ and all $s \in [0, \, T]$ take $k_n(s) \in \{1, \, 2, \, ..., \, 2^{n-1} - 1\}$ such that $s \in [t_{k_n(s)}^n, \, t_{k_n(s)+1}^n ]$. Obviously, if we take $s_n \in [t_{k_n(s)}^n, \, t_{k_n(s)+1}^n ]$ for all $n \in \mathbb{N}$, we will have $s_n \longrightarrow s$ as $n \longrightarrow +\infty$ since $$|s_n - s| \leq |t_{k_n(s)}^n - t_{k_n(s)+1}^n| = \mathcal{O}\left(2^{-n}\right).$$ Denoting by $D$ the set of points $s$ where $\lambda_{s}^{j,n}$ and $b_i\left(s, \, x_1, \, \ldots, \, x_N\right)$ are continuous in $s$ for all $j$ and $n$, for any $s \in D$ and an arbitrary $\epsilon > 0$ we have
\begin{eqnarray}
b_i(s_n, \, \lambda_{s_n}^{1,n}, \, \lambda_{s_n}^{2,n}, \, ..., \, \lambda_{s_n}^{N,n}) - \epsilon \leq b^{i,n}_{k_n(s)} \leq b_i(s, \, \lambda_{s}^{1,n}, \, \lambda_{s}^{2,n}, \, ..., \, \lambda_{s}^{N,n})
\end{eqnarray}
for some $s_n \in [t_{k_n(s)}^n, \, t_{k_n(s)+1}^n]$ (by the definition of infimum). For an $m \in \mathbb{N}$, recalling the pointwise monotonicity of each $\lambda_{\cdot}^{i, n}$ in $n \in \mathbb{N}$ and the monotonicity of each function $b_i$ in each of its last $N$ arguments, the previous double inequality easily gives
\begin{eqnarray}
b_i(s_n, \, \lambda_{s_n}^{1,m}, \, \lambda_{s_n}^{2,m}, \, ..., \, \lambda_{s_n}^{N,m}) - \epsilon \leq b^{i,n}_{k_n(s)} \leq b_i(s, \, \lambda_{s}^{1,n}, \, \lambda_{s}^{2,n}, \, ..., \, \lambda_{s}^{N,n})
\end{eqnarray}
for all $n \geq m$. Since $s \in D$, taking $n \longrightarrow +\infty$ in the above we obtain
\begin{eqnarray}
b_i(s, \, \lambda_{s}^{1,m}, \, \lambda_{s}^{2,m}, \, ..., \, \lambda_{s}^{N,m}) - \epsilon \leq \liminf_{n \longrightarrow +\infty}b^{i,n}_{k_n(s)} \leq \limsup_{n \longrightarrow +\infty}b^{i,n}_{k_n(s)} \leq b_i(s, \, \lambda_{s}^{1}, \, \lambda_{s}^{2}, \, ..., \, \lambda_{s}^{N}) \nonumber \\ .
\end{eqnarray}
Taking now $m \longrightarrow +\infty$ we get
\begin{eqnarray}
b_i(s, \, \lambda_{s}^{1}, \, \lambda_{s}^{2}, \, ..., \, \lambda_{s}^{N}) - \epsilon \leq \liminf_{n \longrightarrow +\infty}b^{i,n}_{k_n(s)} \leq \limsup_{n \longrightarrow +\infty}b^{i,n}_{k_n(s)} \leq b_i(s, \, \lambda_{s}^{1}, \, \lambda_{s}^{2}, \, ..., \, \lambda_{s}^{N}). \nonumber \\
\end{eqnarray}
and since $\epsilon > 0$ was arbitrary, the above implies that $$\displaystyle{\lim_{n \longrightarrow +\infty}b^{i,n}_{k_n(s)} = b_i(s, \, \lambda_{s}^{1}, \, \lambda_{s}^{2}, \, ..., \, \lambda_{s}^{N})},$$ where the convergence is obviously monotone. Next, for any $t \in \left[0, \, T\right]$, recalling \eqref{approxSDE} and that for all $k \in \{1, \, 2, \, ..., \, 2^{n-1} - 1\}$ we have $k = k_n(s)$ for all $s \in \left[t_{k}^n, \, t_{k+1}^n \right]$, for any $i \in \left\{1, \, 2, \, ..., \, N\right\}$ we can write
\begin{eqnarray}\label{af}
\lambda_t^{i,n+1} &=& \lambda_{0}^{i} + \sum_{k=0}^{k_n(t) - 1}\left(\lambda_{t_{k+1}^n}^{i,n+1} - \lambda_{t_{k}^n}^{i,n+1}\right) + \left(\lambda_{t}^{i,n+1} - \lambda_{t_{k_n(t)}^n}^{i,n+1}\right) \nonumber \\
&=& \lambda_{0}^{i} + a_i\sum_{k=0}^{k_n(t) - 1}\int_{t_k^n}^{t_{k+1}^n}\left(\mathbb{E}\left[b^{i,n}_{k_n(s)} | \mathcal{F}_{s}\right] - \lambda_s^{i,n+1}\right)ds \nonumber \\
&& \qquad + a_i\int_{t_{k_n(t)}^n}^{t}\left(\mathbb{E}\left[b^{i,n}_{k_n(s)} | \mathcal{F}_{s}\right] - \lambda_s^{i,n+1}\right)ds + \sum_{k=0}^{k_n(t) - 1}\int_{t_k^n}^{t_{k+1}^n}\sigma_i(\lambda_s^{i,n+1})dW_s^i \nonumber \\
&& \qquad + \int_{t_{k_n(t)}^n}^{t}\sigma_i(\lambda_s^{i,n+1})dW_s^i  + \sum_{k=0}^{k_n(t) - 1}\int_{t_k^n}^{t_{k+1}^n}\int_{U_1}g_{i,1}\left(\lambda_{s-}^{i, n+1}, \, u\right)N_{i,1}\left(ds, \, du\right) \nonumber \\
&& \qquad + \int_{t_{k_n(t)}^n}^{t}\int_{U_1}g_{i,1}\left(\lambda_{s-}^{i, n+1}, \, u\right)N_{i,1}\left(ds, \, du\right) \nonumber \\
&& \qquad + \sum_{k=0}^{k_n(t) - 1}\int_{t_k^n}^{t_{k+1}^n}\int_{U_0}g_{i,0}\left(\lambda_{s-}^{i, n+1}, \, u\right)\tilde{N}_{i,0}\left(ds, \, du\right) \nonumber \\
&& \qquad + \int_{t_{k_n(t)}^n}^{t}\int_{U_0}g_{i,0}\left(\lambda_{s-}^{i, n+1}, \, u\right)\tilde{N}_{i,0}\left(ds, \, du\right) \nonumber \\
&=& \lambda_{0}^{i} + a_i\int_{0}^{t}\left(\mathbb{E}\left[b^{i,n}_{k_n(s)} | \mathcal{F}_{s}\right] - \lambda_s^{i,n+1}\right)ds + \int_{0}^{t}\sigma_i(\lambda_s^{i,n+1})dW_s^i \nonumber \\
&& + \int_{0}^{t}\int_{U_1}g_{i,1}\left(\lambda_{s-}^{i, n+1}, \, u\right)N_{i,1}\left(ds, \, du\right)  + \int_{0}^{t}\int_{U_0}g_{i,0}\left(\lambda_{s-}^{i, n+1}, \, u\right)\tilde{N}_{i,0}\left(ds, \, du\right) \nonumber \\
\end{eqnarray}
and taking $n \longrightarrow +\infty$ in the above for all $i$ we derive the desired system of SDEs satisfied by the limiting processes $\left\{\lambda_{\cdot}^{i}: \, i \in \left\{1, \, 2, \, ..., \, N\right\}\right\}$. Indeed, since $\left[0, \, T\right] \slash D$ is obviously a countable random subset of $\left[0, \, T\right]$, the monotone convergence theorem gives
\begin{eqnarray}
\int_{0}^{t}\left(\mathbb{E}\left[b^{i,n}_{k_n(s)} | \mathcal{F}_{s}\right] - \lambda_s^{i,n+1}\right)ds &=& \int_{0}^{t}\mathbb{E}\left[b^{i,n}_{k_n(s)} | \mathcal{F}_{s}\right]ds - \int_{0}^{t}\lambda_s^{i,n+1}ds \nonumber \\
&\longrightarrow& \int_{0}^{t}\mathbb{E}\left[b_i(s, \,\lambda_{s}^{1}, \, \lambda_{s}^{2}, \, ..., \, \lambda_{s}^{N}) | \mathcal{F}_{s}\right]ds - \int_{0}^{t}\lambda_s^{i}ds \nonumber \\
&=& \int_{0}^{t}\left(b_i(s, \,\lambda_{s}^{1}, \, \lambda_{s}^{2}, \, ..., \, \lambda_{s}^{N}) - \lambda_s^{i}\right)ds,
\end{eqnarray}
uniformly in $t \in [0, T]$. Then, by the Burkholder-Davis-Gundy inequality (see \cite{COH}) we have 
\begin{eqnarray}
&&\mathbb{E}\left[\left(\sup_{t \in \left[0, \, T\right]}\left|\int_0^{t \wedge \tau^m}\sigma_i\left(\lambda_s^{i,n+1}\right)dW^{i}_s - \int_0^{t \wedge \tau^m}\sigma_i\left(\lambda_s^{i}\right)dW^{i}_s \right|\right)^2\right] \nonumber \\
&& \qquad  = \mathbb{E}\left[\left(\sup_{t \in \left[0, \, T\right]}\left|\int_0^{t \wedge \tau^m}\left(\sigma_i\left(\lambda_s^{i,n+1}\right) - \sigma_i\left(\lambda_s^{i}\right)\right)dW^{i}_s\right| \right)^2\right] \nonumber \\
&& \qquad \leq C\mathbb{E}\left[\int_0^{T \wedge \tau^m}\left(\sigma_i\left(\lambda_s^{i,n+1}\right) - \sigma_i\left(\lambda_s^{i}\right)\right)^2ds\right] \nonumber
\end{eqnarray}
and for $j \in \left\{0, 1\right\}$ 
\begin{eqnarray}
&&\mathbb{E}\Bigg[\Bigg(\sup_{t \in \left[0, \, T\right]}\Bigg|\int_0^{t \wedge \tau^m}\int_{U_j}g_{i,j}\left(\lambda_{s-}^{i,n+1}, \, u\right)\tilde{N}_{i,j}\left(ds, \, du\right) \nonumber \\
&& \qquad \qquad \qquad \qquad \qquad \qquad \qquad \qquad \quad - \int_0^{t \wedge \tau^m}\int_{U_j}g_{i,j}\left(\lambda_{s-}^{i}, \, u\right)\tilde{N}_{i,j}\left(ds, \, du\right) \Bigg|\Bigg)^2\Bigg] \nonumber \\
&& \qquad \quad = \mathbb{E}\Bigg[\Bigg(\sup_{t \in \left[0, \, T\right]}\Bigg|\int_0^{t \wedge \tau^m}\int_{U_j}\left(g_{i,j}\left(\lambda_{s-}^{i,n+1}, \, u\right) - g_{i,j}\left(\lambda_{s-}^{i}, \, u\right)\right)\tilde{N}_{i,j}\left(ds, \, du\right)\Bigg| \Bigg)^2\Bigg] \nonumber \\
&& \qquad \quad \leq C\mathbb{E}\Bigg[\int_0^{T \wedge \tau^m}\int_{U_j}\left(g_{i,j}\left(\lambda_{s-}^{i,n+1}, \, u\right) - g_{i,j}\left(\lambda_{s-}^{i}, \, u\right)\right)^2\mu_{i,j}\left(du\right)ds\Bigg] \nonumber
\end{eqnarray}
where we denote by $\lambda_{t-}^i$ the monotone limit of $(\lambda_{t-}^{i, n})_{n \in \mathbb{N}}$ as $n \longrightarrow +\infty$, with the sequence $\left\{\tau^m\right\}_{m \in \mathbb{N}}$ of stopping times selected as in the proof of Lemma~\ref{basiclemma} to ensure that the RHS in the last two estimates is finite for all $n \in \mathbb{N}$, and with these RHS tending to zero by the monotone pointwise convergence of $\lambda^{i,n}_{\cdot}$ to $\lambda^{i}_{\cdot}$, the continuity of $\sigma_i$ and $g_{i,j}$ ($\mu_{i,j}$ - almost everywhere for the second by (5) of $\operatorname{Assum1} (\sigma_i,g_{i,0},g_{i,1},N_{i,0},N_{i,1})$), and by the monotonicity or boundedness of these functions. Finally, by the monotonicity of $g_{i,1}$, working as for the convergence of the drifts we get 
\begin{eqnarray}
\int_0^{t \wedge \tau^m}\int_{U_1}g_{i,1}\left(\lambda^{i,n+1}_{s-}, \, u\right)\mu_{i,1}\left(du\right)ds \longrightarrow \int_0^{t \wedge \tau^m}\int_{U_1}g_{i,1}\left(\lambda^{i}_{s-}, \, u\right)\mu_{i,1}\left(du\right)ds
\end{eqnarray}
uniformly in $t \in \left[0, \, T\right]$ as $n \longrightarrow +\infty$, and combining this with the previous convergence result for the integral with respect to $\tilde{N}_{i,1}$ we deduce that uniformly in $t \in \left[0, \, T\right]$ we have
\begin{eqnarray}
\int_0^{t \wedge \tau^m}\int_{U_1}g_{i,1}\left(\lambda^{i,n+1}_{s-}, \, u\right)N_{i,1}\left(ds, \, du\right) \longrightarrow \int_0^{t \wedge \tau^m}\int_{U_1}g_{i,1}\left(\lambda^{i}_{s-}, \, u\right)N_{i,1}\left(ds, \, du\right) 
\end{eqnarray}
as $n \longrightarrow +\infty$. The desired system of SDEs is obtained by observing that almost surely we have $t = t \wedge \tau^m$ for large enough $m$, and that $(\lambda_{t-}^{i})_{t \in [0, T]}$ is the actual left limit process of $(\lambda_{t}^{i})_{t \in [0, T]}$ since the latter is the uniform limit of $(\lambda_{t}^{i, n})_{t \in [0, T]}$ as $n \longrightarrow +\infty$. 

\noindent {\bf Step 4.} {\it Uniqueness, positivity, integrability and comparison property.} We note now that for every $i$ and all $t \geq 0$ we have almost surely $\lambda_t^i \geq \lambda_t^{i,1}$ with $\lambda_t^{i,1}$ being non-negative, we can integrate \eqref{almostfinalbound} and use Fatou's lemma to deduce that $\lambda_{\cdot}^i$ is $L^1$ - integrable for each $i$, and there no other solution to our problem by Lemma~\ref{Uniqthm}. It remains to establish the comparison property of solutions, for which we only need to consider the approximating sequences $(\lambda_{1,\cdot}^{i, n})_{n \in \mathbb{N}}$ and $(\lambda_{2,\cdot}^{i, n})_{n \in \mathbb{N}}$ to $\lambda_{1,\cdot}^{i}$ and $\lambda_{2,\cdot}^{i}$ respectively and show that $\lambda_{1,t}^{i, n} \geq \lambda_{2,t}^{i, n}$ for all $t \in [0, T]$, all $i \in \{1, \, \ldots, \, N\}$ and all $n \in \mathbb{N}$. The latter can be done by induction on $n$. For $n = 1$, both $(\lambda_{1,t}^{i, 1})_{t \in [0, T]}$ and $(\lambda_{2,t}^{i, 1})_{t \in [0, T]}$ satisfy \eqref{initialsystem}, so since $\lambda_{1,0}^{i, 1} \geq \lambda_{2,0}^{i, 1}$, by the comparison property of Lemma~\ref{basiclemma} we get that $\lambda_{1,t}^{i, 1} \geq \lambda_{2,t}^{i, 1}$ for all $t \in [0, T]$ and all $i \in \{1, \, \ldots, \, N\}$. Then, supposing that $\lambda_{1,t}^{i, n} \geq \lambda_{2,t}^{i, n}$ for all $t \in [0, T]$, all $i \in \{1, \, \ldots, \, N\}$ and some $n \in \mathbb{N}$, for a fixed $i$ and $j \in \{1, 2\}$ we have that 
\begin{eqnarray}\label{twoapproxSDEs}
\lambda_{j,t}^{i,n+1}
&=& \lambda_{j, 0}^{i} + a_i\int_{0}^{t}\left(\mathbb{E}\left[b^{i,n}_{j,k_n(s)} | \mathcal{F}_{s}\right] - \lambda_{j,s}^{i,n+1}\right)ds + \int_{0}^{t}\sigma_i(\lambda_{j,s}^{i,n+1})dW_s^i \nonumber \\
&& + \int_{0}^{t}\int_{U_1}g_{i,1}\left(\lambda_{j,s-}^{i, n+1}, \, u\right)N_{i,1}\left(ds, \, du\right)  + \int_{0}^{t}\int_{U_0}g_{i,0}\left(\lambda_{j, s-}^{i, n+1}, \, u\right)\tilde{N}_{i,0}\left(ds, \, du\right) \nonumber \\
\end{eqnarray}
where
\begin{eqnarray}
b^{i,n}_{j,k} = \inf_{s \in \left[t^n_k, \, t^{n}_{k+1}\right]}b_i(s, \, \lambda_{j,s}^{1,n}, \, \lambda_{j,s}^{2,n}, \, ..., \, \lambda_{j,s}^{N,n}),
\end{eqnarray}
so the monotonicity of $b_i$ implies that $b^{i,n}_{1,k} \geq b^{i,n}_{2,k}$ for all $k \in \{0, \, 1, \, \ldots, \, 2^{n-1} - 1\} $ and thus the comparison property of Lemma~\ref{basiclemma} applied on equations \eqref{twoapproxSDEs} gives $\lambda_{1,t}^{i, n+1} \geq \lambda_{2,t}^{i, n+1}$ for all $t \in [0, T]$, which completes the induction. The proof of the theorem is now complete.
\end{proof}



\section{Appendix: Proofs of auxiliary lemmas}\label{sec: key lemma}
We prove here the two lemmas that were required for establishing Theorem~\ref{thm:1}. We present first a proof for Lemma~\ref{Uniqthm}, which is just a small modification of that of the uniqueness result in \cite{FL2010}. Then, we proceed to the proof of Lemma~\ref{basiclemma}, which requires Lemma~\ref{Uniqthm}.

\begin{proof}[Proof of Lemma 3.2]
Suppose that that the system \eqref{mainsystem} admits two different solutions $\big(\lambda_{\cdot}^{1,1}, \, \lambda_{\cdot}^{1,2}, \, ..., \, \lambda_{\cdot}^{1,N}\big)$ and $\big(\lambda_{\cdot}^{2,1}, \, \lambda_{\cdot}^{2,2}, \, ..., \, \lambda_{\cdot}^{2,N}\big)$. A localization similar to that in the proof of Proposition 2.2 in \cite{FL2010}, i.e working between the jumps of $\int_{0}^{t}\int_{U_1\backslash U_2}N_{i,1}(du, ds)$ for $i \in \{1, \, 2, \, ..., \, N\}$, allows us to assume that $U_1 = U_2$. Let now $\lambda_{\cdot}^{i} = \lambda_{\cdot}^{1,i} - \lambda_{\cdot}^{2,i}$ for each $i \in \{1, \, 2, \, ..., \, N\}$. The idea is to show that each $\lambda_{\cdot}^{i}$ is identically zero. We are going to construct a sequence $\left\{\phi_{k}\right\}_{k \in \mathbb{N}}$ of non-negative and twice continuously differentiable functions, which satisfies:
\begin{enumerate}
    \item $\phi_{k}(x) \longrightarrow |x|$ increasingly as $k \longrightarrow +\infty$.
    \item $0 \leq \phi_{k}'(x) \leq 1$ for $x \geq 0$ and $-1 \leq \phi_{k}'(x) \leq 0$ for $x \leq 0$.
    \item $\phi_{k}''(x) \geq 0$ for all $x \in \mathbb{R}$.
    \item As $k \longrightarrow +\infty$ we have the following two convergences for all $i \in \{1, \, 2, \, ..., \, N\}$, uniformly in $0 \leq x, \, y \leq m$ for any $m \geq 0$:
    \begin{enumerate}
        \item $\displaystyle{\phi_{k}''(x - y)\left[\sigma_i(x) - \sigma_i(y)\right]^2 \longrightarrow 0}$
        \item $\displaystyle{\int_{U_0}D_{g_{i,0}(x, u) - g_{i,0}(y, u)}\phi_{k}(x - y)\mu_{i,0}(du) \longrightarrow 0}$
\end{enumerate}
        where $D_{z}f(x) := f(x + z) - f(x) - zf'(x)$ for any $x, z \in \mathbb{R}$ and any function $f$ defined on a domain containing $x, x+z$ and differentiable at $x$. 

\end{enumerate}
The construction of $\left\{\phi_{k}\right\}_{k \in \mathbb{N}}$ is almost identical to that in the proof of \cite[Theorem 3.2]{FL2010}. First, we pick the sequence $\{a_k\}_{k \in \mathbb{N}}$ such that for each $k \geq 1$ we have $\int_{a_{k}}^{a_{k-1}}\frac{1}{\rho^2(x)}dx = k$, with $\rho(\cdot)$ given in (4) of $\operatorname{Assum1} (\sigma_i,g_{i,0},g_{i,1},N_{i,0},N_{i,1})$ (which can be taken the same for all $i \in \{1, \, 2, \, ..., \, N\}$). Next, for each $k$, we take a smooth function $x \longrightarrow \psi_{k}(x)$ supported in $\left(a_{k}, \, a_{k-1}\right)$ such that:
\begin{eqnarray}\label{boundrho}
0 \leq \psi_{k}(x) \leq \frac{2}{k}\frac{1}{\rho^2(x)} 
\end{eqnarray}
with $\int_{a_{k}}^{a_{k-1}}\psi_{k}(x)dx = 1$, and we define
\begin{eqnarray}
\phi_{k}(z) = \int_0^{|z|}\int_{0}^{y}\psi_{k}(x)dx
\end{eqnarray}
for $z \in \mathbb{R}$. The difference compared to the construction in the proof of \cite[Theorem~3.2]{FL2010} is that $\rho_m(x)$ is replaced by $\rho(x)$, so the whole construction is independent of $m$, but it can be verified in the same way that the functions $\phi_{k}$ here are non-negative, twice continuously differentiable, and they satisfy the first three of the four required conditions. Conditions 4(a) and 4(b) are verified almost exactly as the convergence part of (iii) and (iv) of \cite[Theorem~3.1]{FL2010} respectively in the proof of \cite[Theorem 3.2]{FL2010} (we just use (i) and (ii) of (4) in $\operatorname{Assum1} (\sigma_i,g_{i,0},g_{i,1},N_{i,0},N_{i,1})$ instead of (3.b) in \cite{FL2010}).  

Now we subtract the SDEs satisfied by $\lambda_{\cdot}^{1,i}$ and $\lambda_{\cdot}^{2,i}$ for each $i \in \{1, \, 2, \, ..., \, N\}$ to get
\begin{eqnarray}
\lambda_t^{i} = \lambda_t^{1,i} - \lambda_t^{2,i} &=&  a_i\int_{0}^{t}\left(b_{i}\left(s, \, \lambda_s^{1,1}, \, \lambda_s^{1,2}, \, ..., \, \lambda_s^{1,N}\right) - b_{i}\left(s, \, \lambda_s^{2,1}, \, \lambda_s^{2,2}, \, ..., \, \lambda_s^{2,N}\right)\right)ds \nonumber \\
&& \qquad - a_i\int_{0}^{t}\left(\lambda_s^{1,i} - \lambda_s^{2,i}\right)ds + \int_0^t\left(\sigma_i(\lambda_s^{1,i}) - \sigma_i(\lambda_s^{2,i})\right)dW_s^i \nonumber \\
&& \qquad + \int_0^{t}\int_{U_2}\left(g_{i,1}\left(\lambda_{s-}^{1,i}, \, u\right) - g_{i,1}\left(\lambda_{s-}^{2,i}, \, u\right) \right)N_{i,1}\left(ds, \, du\right) \nonumber \\
&& \qquad + \int_0^{t}\int_{U_0}\left(g_{i,0}\left(\lambda_{s-}^{1,i}, \, u\right) - g_{i,0}\left(\lambda_{s-}^{2,i}, \, u\right) \right)\tilde{N}_{i,0}\left(ds, \, du\right), \nonumber \\
\end{eqnarray}
we set $\tau_m^{i,j} = \inf\left\{t \geq 0: \lambda_t^{j,i} \geq m\right\}$ and $\tau_m = \inf\left\{\tau_m^{i,j}: \,\,(i, j) \in \{1,\,2\} \times \left\{1, \, 2, \, ..., \, N\right\}\}\right\}$, and we apply Ito's formula on $\phi_{k}\left(\lambda_{\cdot}^{i}\right)$ for each $i \in \{1, \, 2, \, ..., \, N\}$ to obtain
\begin{eqnarray}
\phi_{k}\left(\lambda_{t\wedge \tau_m}^{i}\right) &=&  a_i\int_{0}^{t\wedge \tau_m}\phi_{k}'\left(\lambda_{s}^{i}\right)\bigg[b_{i}\left(s, \, \lambda_s^{1,1}, \, ..., \, \lambda_s^{1,N}\right) - b_{i}\left(s, \, \lambda_s^{2,1}, \, ..., \, \lambda_s^{2,N}\right)\bigg]ds \nonumber \\
&& \quad - a_i\int_{0}^{t\wedge \tau_m}\phi_{k}'\left(\lambda_{s}^{i}\right)\left(\lambda_s^{1,i} - \lambda_s^{2,i}\right)ds \nonumber \\
&& \quad + \int_0^{t\wedge \tau_m}\phi_{k}'\left(\lambda_{s}^{i}\right)\left(\sigma_i(\lambda_s^{1,i}) - \sigma_i(\lambda_s^{2,i})\right)dW_s^i \nonumber \\
&& \quad + \frac{1}{2}\int_0^{t\wedge \tau_m}\phi_{k}''\left(\lambda_{s}^{i}\right)\left(\sigma_i(\lambda_s^{1,i}) - \sigma_i(\lambda_s^{2,i})\right)^2ds \nonumber \\
&& \quad + \int_0^{t\wedge \tau_m}\int_{U_2}\Delta_{\left[g_{i,1}\left(\lambda_{s-}^{1,i}, \, u\right) - g_{i,1}\left(\lambda_{s-}^{2,i}, \, u\right)\right]}\phi_{k}\left(\lambda_{s}^{i}\right)\mu_{i,1}\left(du\right)ds \nonumber \\
&& \quad + \int_0^{t\wedge \tau_m}\int_{U_0}D_{\left[g_{i,0}\left(\lambda_{s-}^{1,i}, \, u\right) - g_{i,0}\left(\lambda_{s-}^{2,i}, \, u\right)\right]}\phi_{k}\left(\lambda_{s}^{i}\right)\mu_{i,0}\left(du\right)ds + M_t, \nonumber \\
\end{eqnarray}
where $D_{z}f(x) := f(x + z) - f(x) - zf'(x)$ as defined earlier, $\Delta_{z}f(x) := f(x + z) - f(x)$, and $M_t$ is a martingale. This particular form of Ito's formula was used to derive display (3.3) in the proof of \cite[Theorem 3.1]{FL2010}. The martingale property of the components of $M_t$ can be seen e.g from the discussion in ~\cite[pages 212-213]{Kunita}, since for $j \in \{0, 1\}$ we can write
\begin{eqnarray}
&&\int_0^{t\wedge \tau_m}\int_{U_j}\Delta_{\left[g_{i,j}\left(\lambda_{s-}^{1,i}, \, u\right) - g_{i,j}\left(\lambda_{s-}^{2,i}, \, u\right)\right]}\phi_{k}\left(\lambda_{s}^{i}\right)\tilde{N}_{i,j}\left(ds, \, du\right) \nonumber \\
 &&\qquad = \int_0^{t}\int_{U_j}\mathbbm{1}_{\{s \leq \tau_m\}}\Delta_{\left[g_{i,j}\left(\lambda_{s-}^{1,i}, \, u\right) - g_{i,j}\left(\lambda_{s-}^{2,i}, \, u\right)\right]}\phi_{k}\left(\lambda_{s}^{i}\right)\tilde{N}_{i,j}\left(ds, \, du\right) \nonumber
\end{eqnarray}
and by the 1-Lipschitz property of $\phi_k$ and by (ii) of (4) in $\operatorname{Assum1} (\sigma_i,g_{i,0},g_{i,1},N_{i,0},N_{i,1})$ we can estimate
\begin{eqnarray}
&&\mathbb{E}\left[\int_0^{t}\int_{U_j}\left(\mathbbm{1}_{\{s \leq \tau_m\}}\Delta_{\left[g_{i,j}\left(\lambda_{s-}^{1,i}, \, u\right) - g_{i,j}\left(\lambda_{s-}^{2,i}, \, u\right)\right]}\phi_{k}\left(\lambda_{s}^{i}\right)\right)^2\mu_{i, j}\left(du\right)ds\right] \nonumber \\
 &&\qquad \leq t\rho^2\left(2m\right)\int_{U_j}G_{j}^2(u)\mu_{i, j}\left(du\right)ds \nonumber \\
 && \qquad < \infty \nonumber
\end{eqnarray}
Next, we take expectations on the above, and use the Lipschitz property of $b_i$ and that $\left|\phi_{k}'(x)\right| \leq 1$ to bound the expectations of the first two integrals by 
\begin{eqnarray}
\mathbb{E}\left[\int_{0}^{t\wedge \tau_m}L\sum_{j=1}^{N}\left|\lambda_{s}^{1,j} - \lambda_{s}^{2,j}\right|ds\right] &\leq& L\int_{0}^{t}\sum_{j=1}^{N}\mathbb{E}\left[\left|\lambda_{s\wedge \tau_m}^{1,j} - \lambda_{s\wedge \tau_m}^{2,j}\right|\right]ds \nonumber \\
&=& L\int_{0}^{t}\sum_{j=1}^{N}\mathbb{E}\left[\left|\lambda_{s\wedge \tau_m}^{j}\right|\right]ds
\end{eqnarray}
for some constant $L$, so we get the estimate
\begin{eqnarray}
\mathbb{E}\left[\phi_{k}\left(\lambda_{t\wedge \tau_m}^{i}\right)\right] &\leq&  L\int_{0}^{t}\sup_{1 \leq j \leq N}\mathbb{E}\left[\left|\lambda_{s\wedge \tau_m}^{j}\right|\right]ds \nonumber \\
&& \quad + \mathbb{E}\left[\frac{1}{2}\int_0^{t\wedge \tau_m}\phi_{k}''\left(\lambda_{s}^{i}\right)\left(\sigma_i(\lambda_s^{1,i}) - \sigma_i(\lambda_s^{2,i})\right)^2ds\right] \nonumber \\
&& \quad + \mathbb{E}\left[\int_0^{t\wedge \tau_m}\int_{U_2}\Delta_{\left[g_{i,1}\left(\lambda_{s-}^{1,i}, \, u\right) - g_{i,1}\left(\lambda_{s-}^{2,i}, \, u\right)\right]}\phi_{k}\left(\lambda_{s}^{i}\right)\mu_{i,1}\left(du\right)ds\right] \nonumber \\
&& \quad + \mathbb{E}\left[\int_0^{t\wedge \tau_m}\int_{U_0}D_{\left[g_{i,0}\left(\lambda_{s-}^{1,i}, \, u\right) - g_{i,0}\left(\lambda_{s-}^{2,i}, \, u\right)\right]}\phi_{k}\left(\lambda_{s}^{i}\right)\mu_{i,0}\left(du\right)ds\right]. \nonumber \\
\end{eqnarray}
Then, since $\lambda_{s}^{1,j}$ and $\lambda_{s}^{2,j}$ take values in $\left(0, \, m\right)$ for $s < \tau_m$, they have no jumps larger than $m$ for those values of $s$, so we can use (iii) of (4) in $\operatorname{Assum1} (\sigma_i,g_{i,0},g_{i,1},N_{i,0},N_{i,1})$ but with $g_{i, 1}\wedge m$ replaced by $g_{i, 1}$, and using also the mean value theorem and that $|\phi_{k}'(z)| \leq 1$ for all $z$ we can bound the third integral in the above estimate by $$\int_0^{t \wedge \tau_m}r_m\left(\left|\lambda_{s-}^{1,i} - \lambda_{s-}^{2,i}\right|\right)ds = \int_0^{t \wedge \tau_m}r_m\left(\left|\lambda_{s-}^{i}\right|\right)ds.$$ Thus, taking $k \longrightarrow +\infty$ and using the properties 1. and 4. of $\phi_{k}$ along with the monotone convergence theorem, the last estimate implies:
\begin{eqnarray}
\mathbb{E}\left[\left|\lambda_{t \wedge \tau_m }^{i}\right|\right] &\leq&  L\int_{0}^{t}\sum_{j=1}^N\mathbb{E}\left[\left|\lambda_{s\wedge \tau_m }^{j}\right|\right]ds + \mathbb{E}\left[\int_0^{t \wedge \tau_m}r_m\left(\left|\lambda_{s-}^{i}\right|\right)ds\right] \nonumber \\
&\leq& L\int_{0}^{t}\mathbb{E}\left[\sum_{j=1}^N\left|\lambda_{s\wedge \tau_m }^{j}\right|\right]ds + \mathbb{E}\left[\int_0^{t \wedge \tau_m}r_m\left(\sum_{j=1}^N\left|\lambda_{s- }^{j}\right|\right)ds\right] \nonumber \\
\end{eqnarray}
where the second inequality is obtained by using also the monotonicity of $r_m$. Summing over all $i \in \{1, \, 2, \, ..., \, N\}$ on the left hand side of the above and applying Gronwall's inequality we obtain
\begin{eqnarray}
\mathbb{E}\left[\sum_{j=1}^N\left|\lambda_{t \wedge \tau_m}^{j}\right|\right] &\leq& C(L,N)\mathbb{E}\left[\int_0^{t \wedge \tau_m}r_m\left(\sum_{j=1}^N\left|\lambda_{s- }^{j}\right|\right)ds\right] \nonumber \\,
\end{eqnarray}
for some constant $C(N,L)$ depending only on $L$ and $N$, so now we can conclude that 
\[\sum_{j=1}^N\left|\lambda_{t}^{j}\right| = 0\]
exactly as it is concluded that $\zeta(t) = 0$ in \cite[page 317]{FL2010}. This completes the proof of the lemma.
\end{proof}


\begin{proof}[Proof of Lemma 3.3]
We will split this proof into four different steps.

\noindent {\bf Step 1:} {\it Discretization in time of the process $b$.} For each $n \in \mathbb{N}_+$ we define $t_0^n = 0$, $b_0^n = b_0 - \frac{1}{n}$, and {recursively for $k\in\mathbb N$:}
\begin{eqnarray}
&&t_{k+1}^n = \inf\{t > t_k^n : b_{t_k^n} - \frac{1}{n} > b_t\} \wedge \left(t_k^n + \frac{1}{n}\right) \wedge T, \nonumber\\
&&b_t^n = b_{t_k^n} - \frac{1}{n}, \qquad t \in \mathopen{[\![}t_k^n, \, t_{k+1}^n \mathclose{[\![}. \nonumber 
\end{eqnarray}
Obviously, we have $b_{t}^n \leq b_{t}$ for all $t \in \left[0, \,T\right]$ and $n \in \mathbb{N}$. We also define $b_t^0 = 0$ and  $$\bar{b}_t^n = \left(1 - \frac{1}{n}\right)\times\max\{b_t^m: 0 \leq m \leq n\}$$ for $0 \leq t \leq T$, so we have $0 \leq \bar{b}_{t}^n < \bar{b}_{t}^{n+1} < b_{t}$ for all $n \in \mathbb{N} \cup \{0\}$ and $0 \leq t \leq T$. Next, we construct a similar sequence of piecewise constant processes which approximates $(b_t)_{t \in [0, T]}$ from above by defining $t_{0,n}=0$, $b_{0,n}=b_0+\frac{1}{n}$, and $$\bar{b}_{t,n} = \left(1 + \frac{1}{n}\right)\times\min\{b_{t,m}: 1 \leq m \leq n\}$$ for all $0 \leq t \leq T$ where
\begin{eqnarray}
&&t_{k+1,n}=\inf\{t>t_{k,n}\,:\,b_{t_{k,n}}+\frac{1}{n} < b_t\}\wedge\left(t_{k,n}+\frac{1}{n}\right)\wedge T, \nonumber\\
&&b_{t,n}=b_{t_{k,n}}+\frac{1}{n},\qquad t\in\mathopen{[\![}t_{k,n},t_{k+1,n}\mathclose{[\![}. \nonumber
\end{eqnarray}
In the last construction one has $\bar{b}_{t,n} > \bar{b}_{t,n+1} > b_t$ for all $n \in \mathbb{N}$ and $0 \leq t \leq T$. By definition, for any fixed {positive integer $n$ and $\omega\in\Omega$, the sequences $(t_k^n(\omega))_k$ and $(t_{k,n}(\omega))_k$} are increasing, and we have in addition the following assertion.
\bigskip

\begin{claima}
{For any fixed $\omega\in\Omega$ and $n \in \mathbb{N}$, one has $t_k^n(\omega)=T$ and $t_{k,n}(\omega)=T$ for sufficiently large $k$.
}
\end{claima}
\begin{proof}[Proof of the Claim A]
{We prove this for the sequence $(t_k^n(\omega))_k$ by contradiction. The proof for $(t_{k,n}(\omega))_k$ is almost identical (we just reverse the inequalities and replace $-\frac{1}{n}$ by $\frac{1}{n}$). Suppose that $t_k^n(\omega)$ takes infinitely many values, so then we have}  
\begin{equation}
t_{k+1}^n(\omega) = \inf\{t > t_k^n(\omega) : b_{t_k^n(\omega)}(\omega) - \frac{1}{n} > b_t(\omega)\} \nonumber
\end{equation}
for all large enough $k$. {By} the right continuity of {the process} $b$, we  also have 
\begin{equation}
b_{t_k^n(\omega)}(\omega) - \frac{1}{n} \geq b_{t_{k+1}^n(\omega)}(\omega) \nonumber
\end{equation} 
for all such $k$. Moreover, $t_k^n(\omega)$ increases to a finite limit $t^n(\omega)$ as $k \rightarrow +\infty$, and since {the function} $t\mapsto b_{t}(\omega)$ has a left limit $\ell^n(\omega)$ at $t^n(\omega)$, we  have
\begin{eqnarray}
\ell^n(\omega) = \lim_{k \rightarrow +\infty}b_{t_{k+1}^n(\omega)}(\omega) \nonumber \leq \lim_{k \rightarrow +\infty}b_{t_{k}^n(\omega)}(\omega) - \frac{1}{n} = \ell^n(\omega) - \frac{1}{n} \nonumber
\end{eqnarray}
which is a contradiction. {Therefore, $t_k^n(\omega)$ only takes finitely many values in $[0,T]$ when $k$ varies. In particular, there exists $\ell^n(\omega)\in[0,T]$ and $k_0\in\mathbb N$ such that $t_k^n(\omega)=\ell^n(\omega)$ for any $k\in\mathbb N$ with $k\geq k_0$. Note that $\ell^n(\omega)$ should equal $T$ since otherwise by the right continuity of the process $b$ we would have $t_{k_0+1}^n(\omega)>t_{k_0}^n(\omega)$, which leads again to a contradiction.
} 

\end{proof}

\noindent{\bf Step 2.} {\it Resolution of the equation with discretized drift coefficients.} {Note that} $t_k^n$ is a stopping time for each $n$ and each $k$, and if we define $\bar{t}_k^n$ to be the $k^{\text{th}}$ smallest element of the set $\{t_k^m: \, k \in \mathbb{N}, \, m \in \{1,2,\ldots,n\}\}$, then for any $n$, $\{\bar{t}_k^n\}_{k \in \mathbb{N}}$ is an increasing sequence of stopping times, with $\bar{b}_t^n$ being constant {on each stochastic interval of the form $\mathopen{[\![}\bar{t}_k^n, \,\bar{t}_{k+1}^n\mathclose{[\![}$}. {Moreover, we obtain by Claim A that, for each fixed $\omega\in\Omega$, $\bar{t}_k^n(\omega)=T$} for all large enough $k$. Assuming that we can find a non-negative semimartingale $\left(Y_t^n\right)_{t \in \mathopen{[\![}0, \, \bar{t}_k^n\mathclose{]\!]}}$ satisfying the SDE
\begin{equation}
\begin{split}
Y_t^n = Y_0 &+ a\int_{0}^{t}\left(\bar{b}_s^n - Y_s^n\right)ds + \int_0^t\sigma(Y_s^n)dW_s \\
& + \int_0^{t}\int_{U_1}g_{1}\left(Y_{s-}^n, \, u\right)N_{1}\left(ds, \, du\right)  + \int_0^{t}\int_{U_0}g_{0}\left(Y_{s-}^n, \, u\right)\tilde{N}_{0}\left(ds, \, du\right)\label{auxSDEapprox}
\end{split}
\end{equation}
{on the stochastic interval $\mathopen{[\![}0, \bar{t}_k^n\mathclose{]\!]}$,} we claim that we can extend the solution to the {stochastic interval $\mathopen{[\![}0,  \bar{t}_{k+1}^n\mathclose{]\!]}$.} Indeed, we only need to find a non-negative solution to the SDE
\begin{equation}
\begin{split}
Y_t^n = Y_{\bar{t}_k^n}^n &+ a\int_{\bar{t}_k^n}^{t}\left(\bar{b}_s^n - Y_s^n\right)ds + \int_{\bar{t}_k^n}^t\sigma(Y_s^n)dW_s \\& + \int_{\bar{t}_k^n}^{t}\int_{U_1}g_{1}\left(Y_{s-}^n, \, u\right)N_{1}\left(ds, \, du\right) + \int_{\bar{t}_k^n}^{t}\int_{U_0}g_{0}\left(Y_{s-}^n, \, u\right)\tilde{N}_{0}\left(ds, \, du\right)\label{auxSDEapprox2}
\end{split}
\end{equation}
{on $\mathopen{]\!]}\bar{t}_k^n, \bar{t}_{k+1}^n\mathclose{]\!]}$} given $\mathcal{F}_{\bar{t}_k^n}$, in which case $\bar{t}_k^n$, $Y_{\bar{t}_k^n}^n$ and $\bar{b}_s^n = \bar{b}_{\bar{t}_k^n}^n \geq 0$ are known constants and $\bar{t}_{k+1}^n$ is a stopping time. This is possible by recalling Theorem 5.1 from \cite{FL2010} (see also Remark~\ref{assumrm2} regarding the verification of the required conditions) to solve
\begin{equation}
\begin{split}
Y_t^n = Y_{\bar{t}_k^n}^n& + a\int_{\bar{t}_k^n}^{t}\left(\bar{b}_{\bar{t}_k^n}^n - Y_{s}^n\right)ds + \int_{\bar{t}_k^n}^t\sigma(Y_s^n)dW_s \\& + \int_{\bar{t}_k^n}^{t}\int_{U_1}g_{1}\left(Y_{s-}^n, \, u\right)N_{1}\left(ds, \, du\right) + \int_{\bar{t}_k^n}^{t}\int_{U_0}g_{0}\left(Y_{s-}^n, \, u\right)\tilde{N}_{0}\left(ds, \, du\right)\label{auxSDEapprox3}
\end{split}
\end{equation}
{on $\mathopen{]\!]}\bar{t}_k^n, T\mathclose{]\!]}$} given $\mathcal{F}_{\bar{t}_k^n}$, and then stopping at time $\bar{t}_{k+1}^n$. This inductive argument defines a non-negative c\`adl\`ag semimartingale $Y_{\cdot}^n$ which solves {the equation $\text{\eqref{auxSDE}}_{\bar b^n}$ on} $\left[0, T\right]$. 
By the same argument, we obtain that the equation $\text{\eqref{auxSDE}}_{\bar{b}_{\cdot, n}}$ with the piecewise constant process $\bar{b}_{\cdot, n}$ in the drift admits also a solution, which we denote by $Y_{\cdot, n}$ for each $n$.
\bigskip

\noindent
{\bf Step 3.} {\it Convergence of the  drift coefficients and associated solutions.} We begin with the following claim. 
\begin{claimb}
The sequences $(Y^n)_{n\geq 1}$ and $(Y_{\cdot,n})_{n\geq 1}$ defined in Step 2 converge pointwise to two $\mathbb F$-adapted processes $Y$ and $\tilde{Y}$, increasingly from below and decreasingly from above respectively. 
\end{claimb}
\begin{proof}[Proof of Claim B]
{By} construction we have $$\bar{b}_{t}^{n} > \bar{b}_{t}^{n+1} > b_t > \bar{b}_{t}^{n+1} > \bar{b}_{t}^{n} \geq 0$$ for all $t \in \left[0, \,T\right]$ and $n \in \mathbb{N}$, and the processes $(\bar{b}_{t}^{n})_{t\in[0,T]}$, $(\bar{b}_{t}^{n+1})_{t\in[0,T]}$, $(\bar{b}_{t,n})_{t\in[0,T]}$ and $(\bar{b}_{t, n+1})_{t\in[0,T]}$ are constant on the stochastic intervals $\mathclose{[\![}\bar{t}_k^n, \bar{t}_{k+1}^n\mathopen{[\![}$, $\mathclose{[\![}\bar{t}_k^{n+1}, \bar{t}_{k+1}^{n+1}\mathopen{[\![}$, $\mathclose{[\![}\bar{t}_{k,n}, \bar{t}_{k+1, n}\mathopen{[\![}$ and $\mathclose{[\![}\bar{t}_{k,n+1}, \bar{t}_{k+1,n+1}\mathopen{[\![}$ respectively for every $k$. Then, denoting by $\tau_{1}^n$ the smallest time in the set $\cup_{k \geq 1}\{\bar{t}_k^n, \, \bar{t}_k^{n+1}, \, \bar{t}_{k,n}, \, \bar{t}_{k,n+1}\}$, by $\tau_{2}^n$ the second smallest and so on, we obtain an increasing sequence of stopping times $\{\tau_k^n\}_{k=1}^{\infty}$, with the processes $(\bar{b}_{t}^{n})_{t\in[0,T]}$, $(\bar{b}_{t}^{n+1})_{t\in[0,T]}$, $(\bar{b}_{t,n})_{t\in[0,T]}$ and $(\bar{b}_{t,n+1})_{t\in[0,T]}$ being constant on each $\mathclose{[\![}\tau_k^n, \tau_{k+1}^n\mathopen{[\![}$. Moreover, making all these processes constant on the interval $\mathclose{[\![}\tau_k^n, T\mathclose{]\!]}$ by ignoring the jumps they admit after $\tau_{k}^n$, the solutions $(Y_t^n)_{t \in [0, T]}$, $(Y_t^{n+1})_{t \in [0, T]}$, $(Y_{t,n})_{t \in [0, T]}$ and $(Y_{t,n+1})_{t \in [0, T]}$ to $\text{\eqref{auxSDE}}_{\bar b^n}$, $\text{\eqref{auxSDE}}_{\bar b^{n+1}}$, $\text{\eqref{auxSDE}}_{\bar b_{\cdot,n}}$ and $\text{\eqref{auxSDE}}_{\bar b_{\cdot,n+1}}$ respectively remain unchanged on the closed stochastic interval $\mathclose{[\![}\tau_k^n, \tau_{k+1}^n\mathclose{]\!]}$, so applying the comparison theorem of Gal’chuk \cite[Theorem 1]{Y1986} on $\mathclose{[\![}\tau_k^n, T\mathclose{]\!]}$ given $F_{\tau_k^n}$ (see Remark~\ref{assumrm3} regarding the verification of all the required conditions) we get that the ordering $$Y_{\tau_k^{n}, n} \geq Y_{\tau_k^n, n+1} \geq Y_{\tau_k^{n}}^{n+1} \geq Y_{\tau_k^n}^n$$ implies the ordering 
\begin{equation*}
Y_{t, n} \geq Y_{t, n+1} \geq Y_{t}^{n+1} \geq Y_{t}^n
\end{equation*}
for all $t \in \mathclose{[\![}\tau_k^n, \tau_{k+1}^n\mathclose{]\!]}$. Therefore, we can extend the last quadruple inequality from each stochastic interval $\left[0, \,\tau_k^n\right]$ to the next one $\left[0, \,\tau_{k+1}^n\right]$, so an inductive argument shows that 
\begin{equation}\label{quadineq}
Y_{t, n} \geq Y_{t, n+1} \geq Y_{t}^{n+1} \geq Y_{t}^n \geq 0
\end{equation}
holds for all $t \in \left[0, \,T\right]$ and $n \in \mathbb{N}$.
Therefore, the sequences $(Y_t^n)_{n\geq 1}$ and $(Y_{t,n})_{n\geq 1}$ are both monotone and bounded, which means that they converge for each $t \in [0, T]$, defining respectively the limiting processes $(Y_t)_{t \in [0, T]}$ and $(\tilde{Y}_t)_{t \in [0, T]}$ which are clearly $\mathbb F$-adapted and satisfy
\begin{equation}\label{squeeze}
Y_{t, 1} \geq \tilde{Y}_{t} \geq Y_{t} \geq Y_{t}^1 \geq 0
\end{equation}
for all $t \in [0, T]$.
\end{proof}

\noindent {We now show that for any $\omega\in\Omega$ and any continuity point $t$ of the function $s\mapsto b_s(\omega)$, the sequences $(\bar{b}_{t}^n(\omega))_{n \geq 1}$ and $(\bar{b}_{t,n}(\omega))_{n \geq 1}$ converge to $b_{t}(\omega)$ from below and above respectively as $n \rightarrow +\infty$. Will will only prove this for $(\bar{b}_{t}^n(\omega))_{n \geq 1}$ as the proof for $(\bar{b}_{t,n}(\omega))_{n \geq 1}$ is identical (we only need to replace $b_t^n(\omega)$ with $b_{t,n}(\omega)$, $\bar{b}_t^n(\omega)$ with $\bar{b}_{t,n}(\omega)$ and $(t_k^n)_{k,n}$ with $(t_{k,n})_{k,n}$ and reverse \eqref{bsordering} in the following argument, and also replace $-\frac{1}{n}$ with $\frac{1}{n}$ in \eqref{convbt} and \eqref{bsordering} below). For any any positive integer $n$,} there exists a $k(n) \in \mathbb{N}$ such that $t \in \mathopen{[\![}t_{k(n)}^n(\omega),  t_{k(n)+1}^n(\omega)\mathclose{[\![}$ and thus 
\begin{eqnarray}
|t - t_{k(n)}^n(\omega)| \leq |t_{k(n)+1}^n(\omega) - t_{k(n)}^n(\omega)| \leq \frac{1}{n},
\end{eqnarray}
which means that $t_{k(n)}^n(\omega) \rightarrow t$ from below as $n \rightarrow +\infty$. Hence, by the continuity of {$b_{\cdot}(\omega)$} at $t$ and the definition of $b_{\cdot}^n$, we have 
\begin{equation}\label{convbt}
b_t^n(\omega) = b_{t_{k(n)}^n(\omega)}(\omega) - \frac{1}{n} \longrightarrow b_{t}(\omega)\text{ as $n \rightarrow +\infty$}.
\end{equation} Recalling then that 
\begin{equation}\label{bsordering}
b_t(\omega) \geq \bar{b}_t^n(\omega) \geq \left(1 - \frac{1}{n}\right)\times b_t^n(\omega)
\end{equation}
for all $n \in \mathbb{N}$, we deduce that $\bar{b}_t^n(\omega) \rightarrow b_t(\omega)$ as $n \rightarrow +\infty$.
\bigskip

\noindent {\bf Step 4.} {\it Resolution of the initial equation via monotone approximations.}
Finally, we will show that the processes $Y$ and $\tilde{Y}$ solve \eqref{auxSDE} in $\left[0, \, T\right]$, which means that they will coincide by Theorem~\ref{Uniqthm}, and we will also establish the comparison property. Once again, we only obtain our SDE for $Y$, as the proof for $\tilde{Y}$ is identical (we just reverse the monotonicity). {For $s\in[0,T]$, we denote by $Y_{s-}$ the limit of the increasing sequence $(Y_{s-}^n)_{n\in\mathbb N}$.} 
We start by recalling the monotone convergence theorem which gives
\begin{eqnarray}
\int_{0}^{t}\left(\bar{b}_s^n - Y_s^n\right)ds &=& \int_{0}^{t}\bar{b}_s^nds - \int_{0}^{t}Y_s^nds \nonumber \\
&\longrightarrow& \int_{0}^{t}b_sds - \int_{0}^{t}Y_sds = \int_{0}^{t}\left(b_s - Y_s\right)ds
\end{eqnarray}
as $n \longrightarrow +\infty$, uniformly in $t \in \left[0, \, T\right]$. Next, we pick a sequence $\{\tau^{m}\}_{m \in \mathbb{N}}$ of $\mathbb{F}$-stopping times such that $\displaystyle{\lim_{m \longrightarrow +\infty}\tau^{m} = +\infty}$ and also
\begin{eqnarray}
&&\int_0^{T \wedge \tau^{m}}\left(\sigma\left(Y_s^n\right) - \sigma\left(Y_s\right)\right)^2ds \leq m, \nonumber \\
&& \int_0^{T \wedge \tau^{m}}\int_{U_0}\left(g_{0}\left(Y_{s-}^n, \, u\right) - g_{0}\left(Y_{s-}, \, u\right)\right)^2\mu_{0}\left(du\right)ds \leq m, \nonumber \\
&& \int_0^{T \wedge \tau^{m}}\int_{U_1}\left(g_{1}\left(Y_{s-}^n, \, u\right) - g_{1}\left(Y_{s-}, \, u\right)\right)^2\mu_{1}\left(du\right)ds \leq m
\end{eqnarray}
and 
\begin{eqnarray}
\int_0^{T \wedge \tau^{m}}\int_{U_1}\left|g_{1}\left(Y_{s-}^n, \, u\right) - g_{1}\left(Y_{s-}, \, u\right)\right|\mu_{1}\left(du\right)ds \leq m
\end{eqnarray}
for all $m \in \mathbb{N}$ and $n \in \mathbb{N}$. The last is possible by using (ii) of (4) in $\operatorname{Assum1}(\sigma,g_0,g_1,N_0,N_1)$ (combined with (ii) of (3) for controlling the fourth integral) and the monotonicity of both the function $\rho(\cdot)$ and the sequence $Y^n$. Then, by using the Burkholder-Davis-Gundy inequality (see \cite{COH}) we have 
\begin{eqnarray}
&&\mathbb{E}\left[\left(\sup_{t \in \left[0, \, T\right]}\left|\int_0^{t \wedge \tau^m}\sigma\left(Y_s^n\right)dW_s - \int_0^{t \wedge \tau^m}\sigma\left(Y_s\right)dW_s \right|\right)^2\right] \nonumber \\
&& \qquad = \mathbb{E}\left[\left(\sup_{t \in \left[0, \, T\right]}\left|\int_0^{t \wedge \tau^m}\left(\sigma\left(Y_s^n\right) - \sigma\left(Y_s\right)\right)dW_s\right| \right)^2\right] \nonumber \\
&& \qquad \leq C\mathbb{E}\left[\int_0^{T \wedge \tau^m}\left(\sigma\left(Y_s^n\right) - \sigma\left(Y_s\right)\right)^2ds\right] \nonumber
\end{eqnarray}
where we can recall the continuity of $\sigma$ and either the monotone convergence theorem or the dominated convergence theorem (depending on whether $\sigma$ is bounded or increasing) to deduce that the RHS tends to zero as $n \longrightarrow +\infty$. Next, writing $\tilde{N}_{1}\left(ds,du\right)$ for the compensated measure $N_{1}\left(ds,du\right) - \mu_{1}\left(du\right)ds$, where $\mu_{1}\left(du\right)ds$ is the compensator of $N_{1}\left(ds,du\right)$, by using the Burkholder-Davis-Gundy inequality once more we have
\begin{eqnarray}
&&\mathbb{E}\Bigg[\Bigg(\sup_{t \in \left[0, \, T\right]}\Bigg|\int_0^{t \wedge \tau^m}\int_{U_1}g_{1}\left(Y_{s-}^n, \, u\right)\tilde{N}_{1}\left(ds, \, du\right) \nonumber \\
&& \qquad \qquad \qquad - \int_0^{t \wedge \tau^m}\int_{U_1}g_{1}\left(Y_{s-}, \, u\right)\tilde{N}_{1}\left(ds, \, du\right) \Bigg|\Bigg)^2\Bigg] \nonumber \\
&& \qquad  = \mathbb{E}\Bigg[\Bigg(\sup_{t \in \left[0, \, T\right]}\Bigg|\int_0^{t \wedge \tau^m}\int_{U_1}\left(g_{1}\left(Y_{s-}^n, \, u\right) - g_{1}\left(Y_{s-}, \, u\right)\right)\tilde{N}_{1}\left(ds, \, du\right)\Bigg| \Bigg)^2\Bigg] \nonumber \\
&& \qquad  \leq C\mathbb{E}\Bigg[\int_0^{T \wedge \tau^m}\int_{U_1}\left(g_{1}\left(Y_{s}^n, \, u\right) - g_{1}\left(Y_{s}, \, u\right)\right)^2\mu_{1}\left(du\right)ds\Bigg] \nonumber
\end{eqnarray}
where the quantity $\left(g_{1}\left(Y_{s}^n, \, u\right) - g_{1}\left(Y_{s}, \, u\right)\right)^2$ decreases to zero for $\mu_1$ - almost all $u$ (by the monotonicity of $g_1$ and (5) of $\operatorname{Assum1}(\sigma,g_0,g_1,N_0,N_1)$), so by monotone convergence we get that the RHS of the above tends also to zero as $n \longrightarrow +\infty$. Moreover, by the monotonicity of $g_1$, we can work as for the convergence of the drifts to obtain
\begin{eqnarray}
\int_0^{t \wedge \tau^m}\int_{U_1}g_{1}\left(Y_{s}^n, \, u\right)\mu_{1}\left(du\right)ds \longrightarrow \int_0^{t \wedge \tau^m}\int_{U_1}g_{1}\left(Y_{s}, \, u\right)\mu_{1}\left(du\right)ds
\end{eqnarray}
uniformly in $t \in \left[0, \, T\right]$ as $n \longrightarrow +\infty$, and combining this with the previous convergence result we deduce that almost surely we have
\begin{eqnarray}
\int_0^{t \wedge \tau^m}\int_{U_1}g_{1}\left(Y_{s-}^n, \, u\right)N_{1}\left(ds, \, du\right) \longrightarrow \int_0^{t \wedge \tau^m}\int_{U_1}g_{1}\left(Y_{s-}, \, u\right)N_{1}\left(ds, \, du\right) 
\end{eqnarray}
uniformly in $t \in \left[0, \, T\right]$ (in a subsequence). Finally, using the Burkholder-Davis-Gundy inequality and the monotone convergence theorem as we did for the integral with respect to $\tilde{N}_{1}\left(ds, \, du\right)$, we find that 
\begin{eqnarray}
\mathbb{E}\Bigg[\bigg(\sup_{t \in \left[0, \, T\right]}\bigg|\int_0^{t \wedge \tau^m}\int_{U_0}g_{0}\left(Y_{s-}^n, \, u\right)\tilde{N}_{0}\left(ds, \, du\right)  - \int_0^{t \wedge \tau^m}\int_{U_0}g_{0}\left(Y_{s-}, \, u\right)\tilde{N}_{0}\left(ds, \, du\right) \bigg|\bigg)^2\Bigg] \nonumber
\end{eqnarray}
tends also to zero as $n \longrightarrow +\infty$. It follows that almost surely, we can take limits on both sides of \eqref{auxSDEapprox} and obtain $\eqref{auxSDE}_{b}$ with $t$ is replaced by $t \wedge \tau^m$, for any $t \in \left[0, \, T\right]$, where $(Y_{t-})_{t \in [0, T]}$ is the actual left limit process of $(Y_{t})_{t \in [0, T]}$ since the latter is the uniform limit of $(Y_{t}^{n})_{t \in [0, T]}$ as $n \longrightarrow +\infty$. Then, we can get rid of $\tau^m$ by letting $m \longrightarrow +\infty$. \bigskip

\noindent We finish this proof by establishing the comparison property of solutions. As we have already mentioned, replacing $Y$ with $\tilde{Y}$ and $\bar{b}^n$ with $\bar{b}_{\cdot, n}$ in the above arguments, one finds that $\tilde{Y}$ is also a solution to \eqref{auxSDE}, so by Lemma~\ref{Uniqthm} we get $Y \equiv \tilde{Y}$ which is approximated by $(Y^n)_{n \geq 1}$ and $(Y_{\cdot, n})_{n \geq 1}$ from below and from above respectively. Let now $(Y_t^1)_{t \in [0, T]}$ and $(Y_t^2)_{t \in [0, T]}$ be solutions to $\text{\eqref{auxSDE}}_{\bar b^1}$ and $\text{\eqref{auxSDE}}_{\bar b^2}$ respectively with $Y_0^1 \geq Y_0^2$ and $b_t^1 \geq b_t^2$ for all $t$. We construct the corresponding sequences of piecewise constant $\mathbb{F}$-adapted c\`adl\`ag processes $(\bar{b}_{\cdot, n}^1)_{n \geq 1}$ and $(\bar{b}^{n,2})_{n \geq 1}$, which approximate $b^1$ and $b^2$ from above and below respectively, and we consider the solutions $Y_{\cdot, n}^{1}$ and $Y^{2, n}$ to $\text{\eqref{auxSDE}}_{\bar{b}_{\cdot, n}^1}$ and $\text{\eqref{auxSDE}}_{\bar{b}^{n,2}}$ respectively, with initial values $Y_0^1$ and $Y_0^2$ respectively. Next, for each $n \in \mathbb{N}$, we pick a sequence $(\tau^{k,n})_{k \geq 0}$ of stopping times with $\tau^{0,n} = 0$ and with both processes $(\bar{b}_{t, n}^1)_{t \in [0, T]}$ and $(\bar{b}_t^{n,2})_{t \in [0, T]}$ being constant on each stochastic interval $\mathclose{[\![}\tau^{k,n}, \tau^{k+1,n}\mathopen{[\![}$. Then, we have $\bar{b}_{t, n}^1 > b_t^1 \geq b_t^2 > \bar{b}_t^{n,2}$ for every $t$, so we can apply the comparison theorem of Gal’chuk on each stochastic interval (as we did in the proof of Claim B for the SDEs $\text{\eqref{auxSDE}}_{\bar b^n}$, $\text{\eqref{auxSDE}}_{\bar b^{n+1}}$, $\text{\eqref{auxSDE}}_{\bar b_{\cdot,n}}$ and $\text{\eqref{auxSDE}}_{\bar b_{\cdot,n+1}}$) to deduce that $Y_{t,n}^1 \geq Y_t^{2,n}$ for all $t \in [0, T]$, almost surely. Taking $n \to +\infty$ in the last inequality we conclude that $Y_{t}^1 \geq Y_t^{2}$ for all $t \in [0, T]$, almost surely, since $Y_{t,n}^1 \longrightarrow Y_{t}^1$ and $Y_t^{2,n} \longrightarrow Y_t^{2}$ as $n \longrightarrow +\infty$.
\end{proof}

{\flushleft{\textbf{Acknowledgement}}\\[.1in]
The second-named author's work was supported financially by the Boya Postdoctoral Fellowship of Peking University, and by the Beijing International Center for Mathematical Research (BICMR).
}

\end{document}